\documentclass{scrartcl}

\usepackage{amsfonts} 
\usepackage{amsmath} 
\usepackage{amssymb} 
\usepackage{amsthm} 
\usepackage[english]{babel} 
\usepackage{dsfont} 
\usepackage[T1]{fontenc} 
\usepackage{graphicx} 
\usepackage[ansinew]{inputenc} 
\usepackage{multicol} 
\usepackage{scrpage2} 
\usepackage{tabularx} 
\usepackage[final]{pdfpages} 
\usepackage{rotating}
\usepackage{multicol}
\usepackage[all,knot,arc,poly]{xy}

\newcommand{\betrag}[1]{\left|#1\right|}
\newcommand{\gklammer}[1]{\left\{#1\right\}}

\newcommand{\eklammer}[1]{\left[#1\right]}
\newcommand{\sklammer}[1]{\left\langle#1\right\rangle}

\newtheorem{Lem}{Lemma}[section]
\newtheorem{Thm}{Theorem}[section]
\newtheorem{Cor}{Corollary}[section]
\newtheorem{Prop}{Proposition}
\theoremstyle{definition}
\newtheorem{Rem}{Remark}

\theoremstyle{remark}

\newcommand{\rd}{\ensuremath{\mathrm{d}}} 
\newcommand{\id}{\ensuremath{\mathrm{id}}} 
\newcommand{\Aut}{\ensuremath{\mathrm{Aut}}} 
\newcommand{\op}{\ensuremath{\mathrm{op}}} 
\newcommand{\cop}{\ensuremath{\mathrm{cop}}} 
\newcommand{\e}{\ensuremath{\mathrm{e}}} 
\newcommand{\ot}{\otimes}
\newcommand{\ds}{\mathds}
\newcommand{\cC}{\mathds{C}}
\newcommand{\rR}{\mathds{R}}
\newcommand{\calA}{\mathcal{A}}
\newcommand{\calC}{\mathcal{C}}
\newcommand{\calB}{\mathcal{B}}
\newcommand{\calH}{\mathcal{H}}
\newcommand{\eins}{\mathds{1}}

\newcommand{\mubeta}{
\xy
(-5,0)*{}; (31,0)*{}; 
(-1,14)*{}; (11,14)*{} **\dir{-};
(-1,0)*{}; (11,0)*{} **\dir{-};
(16,14)*{}; (28,14)*{} **\dir{-};
(16,0)*{}; (28,0)*{} **\dir{-};
\vcross~{(0,13)}{(5,13)}{(0,9)}{(5,9)};
(0,14)*{}; (0,13)*{} **\dir{-};
(5,14)*{}; (5,13)*{} **\dir{-};
\vcross~{(5,9)}{(10,9)}{(5,5)}{(10,5)};
(10,5)*{}; (10,0)*{} **\dir{-};
(2.5,0)*{}; (2.5,2)*{} **\dir{-};
(0,4)*{};(5,4)*{} **\crv{(0,2)&(2.5,2)&(5,2)};
(2.5,3)*{\cdot};
(0,4)*{}; (0,9)*{} **\dir{-};
(10,9)*{}; (10,14)*{} **\dir{-};
(5,4)*{}; (5,5)*{} **\dir{-};
(13.5,7)*{=};
(17,14)*{}; (17,6)*{} **\dir{-};
\vcross~{(17,6)}{(24.5,6)}{(17,1)}{(24.5,1)};
(17,1)*{}; (17,0)*{} **\dir{-};
(24.5,1)*{}; (24.5,0)*{} **\dir{-};
(24.5,6)*{}; (24.5,11)*{} **\dir{-};
(22,13)*{};(27,13)*{} **\crv{(22,11)&(24.5,11)&(27,11)}; 
(24.5,12)*{\cdot};
(27,13)*{}; (27,14)*{} **\dir{-};
(22,13)*{}; (22,14)*{} **\dir{-};
\endxy}

\newcommand{\mubetaa}{
\xy
(-5,0)*{}; (31,0)*{}; 
(-1,14)*{}; (11,14)*{} **\dir{-};
(-1,0)*{}; (11,0)*{} **\dir{-};
(16,14)*{}; (28,14)*{} **\dir{-};
(16,0)*{}; (28,0)*{} **\dir{-};
(0,14)*{}; (0,9)*{} **\dir{-};
(5,14)*{}; (5,13)*{} **\dir{-};
(10,14)*{}; (10,13)*{} **\dir{-};
\vcross~{(5,13)}{(10,13)}{(5,9)}{(10,9)};
\vcross~{(0,9)}{(5,9)}{(0,5)}{(5,5)};
(10,9)*{}; (10,4)*{} **\dir{-};
(5,5)*{}; (5,4)*{} **\dir{-};
(0,5)*{}; (0,0)*{} **\dir{-};
(5,4)*{};(10,4)*{} **\crv{(5,2)&(7.5,2)&(10,2)}; 
(7.5,3)*{\cdot};
(7.5,2)*{}; (7.5,0)*{} **\dir{-};
(13.5,7)*{=};
(17,13)*{};(22,13)*{} **\crv{(17,11)&(19.5,11)&(22,11)}; 
(19.5,12)*{\cdot};
(17,13)*{}; (17,14)*{} **\dir{-};
(22,13)*{}; (22,14)*{} **\dir{-};
(27,14)*{}; (27,6)*{} **\dir{-};
(19.5,11)*{}; (19.5,6)*{} **\dir{-};
\vcross~{(19.5,6)}{(27,6)}{(19.5,1)}{(27,1)};
(19.5,1)*{}; (19.5,0)*{} **\dir{-};
(27,1)*{}; (27,0)*{} **\dir{-};
\endxy}

\newcommand{\einsbeta}{
\xy
(-5,0)*{}; (22,0)*{}; 
(-1,14)*{}; (6,14)*{} **\dir{-};
(-1,0)*{}; (6,0)*{} **\dir{-};
(11,14)*{}; (18,14)*{} **\dir{-};
(11,0)*{}; (18,0)*{} **\dir{-};
(5,14)*{}; (5,0)*{} **\dir{-};
(0,0)*{}; (0,7)*{} **\dir{-};
(0,7.5)*{\circ};
(8.5,7)*{=};
(12,14)*{}; (12,5)*{} **\dir{-};
(17,9)*{}; (17,5)*{} **\dir{-};
(17,9.5)*{\circ};
\vcross~{(12,5)}{(17,5)}{(12,1)}{(17,1)};
(12,1)*{}; (12,0)*{} **\dir{-};
(17,1)*{}; (17,0)*{} **\dir{-};
\endxy}

\newcommand{\einsbetaa}{
\xy
(-5,0)*{}; (22,0)*{}; 
(-1,14)*{}; (6,14)*{} **\dir{-};
(-1,0)*{}; (6,0)*{} **\dir{-};
(11,14)*{}; (18,14)*{} **\dir{-};
(11,0)*{}; (18,0)*{} **\dir{-};
(0,14)*{}; (0,0)*{} **\dir{-};
(5,0)*{}; (5,7)*{} **\dir{-};
(5,7.5)*{\circ};
(8.5,7)*{=};
(17,14)*{}; (17,5)*{} **\dir{-};
(12,9)*{}; (12,5)*{} **\dir{-};
(12,9.5)*{\circ};
\vcross~{(12,5)}{(17,5)}{(12,1)}{(17,1)};
(12,1)*{}; (12,0)*{} **\dir{-};
(17,1)*{}; (17,0)*{} **\dir{-};
\endxy}

\begin{document}


\title{Additive Deformations of Braided Hopf Algebras}

\author{Malte Gerhold\\
\emph{\small malte.gerhold@uni-greifswald.de}\\
Stefan Kietzmann\\
\emph{\small stefan.kietzmann.stud@uni-greifswald.de}\\
Stephanie Lachs\\
\emph{\small sl041036@uni-greifswald.de}
}

\date{January 31, 2011}

\maketitle



\begin{abstract}
Additive deformations of bialgebras in the sense of \cite{Wirth}, i.e.\ deformations of the multiplication map fulfilling a certain compatibility condition w.r.t.\ the coalgebra structure, can be generalized to braided bialgebras. The theorems for additive deformations of Hopf algebras can also be carried over to that case. We consider $*$-structures and prove a general Schoenberg correspondence in this context. Finally we give some examples.
\end{abstract}


\section{Introduction}
Consider the $*$-algebra $\calA$ generated by $a$, $a^*$ and $\eins$ with the relation $\eklammer{a,a^*}=\eins$. One wants to define a comultiplication on the generators by
\begin{align*}
\Delta(a)=a\ot\eins+\eins\ot a,\quad
\Delta(a^*)=a^*\ot\eins+\eins\ot a^*,
\end{align*}
but this cannot be extended as an algebra homomorphism from $\calA$ to $\calA \ot \calA$ since the relation is not respected. It is however easily checked that it can be extended as an algebra homomorphism from $\calA$ to $\calA_t\ot \calA_s$ with $t, s \in \rR$ and $t+s=1$. Here $\calA_r$ denotes the algebra with the same generators as $\calA$ but under relation $\eklammer{a,a^*}_r=r\eins$ for $r\in \ds R$. One observes that all $\calA_t$ are defined on the same vector space and $\calA_0$ is just the polynomial $*$-bialgebra in two commuting adjoint indeterminates. This observation gave rise to the definition of additive deformations of $*$-bialgebras in \cite{Wirth} as a family of multiplications $(\mu_t)_{t\in\ds R}$ on some $*$-coalgebra $\calB$ s.t.\ $(\calB,\mu_t)$ is a unital $*$-algebra with the same unit and involution for all $t \in \rR$ and where the main feature is, that the comultiplication $\Delta$ is an $*$-algebra homomorphism from $(\calB,\mu_{t+s})$ to $(\calB,\mu_t)\ot(\calB,\mu_s)$. Particularly, $(\calB,\mu_0)$ is a $*$-bialgebra. It was shown in \cite{Wirth} that all additive deformations are of the form $\mu_t=\mu_0\star \e_\star^{tL}$, where $L$ is a linear functional on $\calB \ot \calB$, and all linear functionals on a bialgebra $\calB$ s.t.\ this equation defines an additive deformation were characterized. J. \textsc{Wirth} also showed that a pointwise continuous convolution semigroup $(\varphi_t)_{t\geq0}$ on $\calB$ with generator\footnote{That means $\varphi_t=\e_\star^{t\psi}$.} $\psi$ consists of states $\varphi_t$ on $\calB_t$ iff $\psi$ is $L$-conditionally positive in the sense that $\psi(a^*a) + L(a^*\ot a) \geq 0$ for all $a \in \ker \delta$. This generalizes the Schoenberg correspondence in the sense of \cite{schurman}.

In \cite{Diplomarbeit} quantum Lévy processes on additive deformations of $*$-bialgebras were constructed and for the additive deformation on $\mathds{C}\eklammer{x,x^*}$, discussed in the beginning, this resulted in a pair of operator processes fulfilling canonical commutation relations.

If one wants to mimic the constructions for the algebra with two adjoint anti-commu-ting generators, there is the problem that this is not a $*$-bialgebra, but a graded $*$-bialgebra (see e.g.\ \cite{schurman}). In this paper we generalize the definition of an additive deformation even to the case of a braided bialgebra in the sense of \cite{Majid} (resp.\ \cite{Franz} for $*$-bialgebras). However, we do not work with braided tensor categories, but it is sufficient in our context to define braided vector spaces as in \cite{Ufer}, because we are concerned only with tensor powers of the same vector space. For more clarity we use a graphic calculus, in the literature known as braid diagramms.

We show how the generator calculus and the Schoenberg correspondence of \cite{Wirth} can be carried over to braided $*$-bialgebras (see section \ref{secgenerator} resp.\ \ref{secSchoen}). Whereas most results can be proved along the lines of \cite{Wirth} for the bialgebra case, the diagrammatic approach gives more insight into the structure of these proofs and things get more involved in the $*$-bialgebra case, where we use the definition of a braided $*$-bialgebra of \cite{Franz}.
Our version of the Schoenberg correspondence (theorem \ref{Sa:Schoenberg_def}) generalizes and unifies two older versions: 
\begin{itemize}
\item In \cite{FSS} there are no additive deformations allowed. 
\item In \cite{Wirth} additive deformations are considered, but no braidings. 
\end{itemize}

In \cite{Preprint} the case of $\calB$ being a Hopf algebra was considered. A Hopf algebra is a bialgebra with antipode, i.e.\ a linear map $S: \calB \to \calB$ which is inverse to the identity map w.r.t.\ convolution. In other words one has $\mu\circ(S\ot \id)\circ\Delta=\mu\circ(\id\ot S)\circ\Delta=\eins\delta$. It was shown that for an additive deformation $(\mu_t)_{t\in\ds R}$ there are so called deformed antipodes $S_t$, which are inverse to the identity w.r.t.\ the convolution corresponding to $\mu_t$. 
This is still true in the braided case, as we show in section \ref{secHopf}.

Finally we present as an example an additive deformation of a braided Hopf $*$-algebra generated by two adjoint, anti-commuting indeterminates (see section \ref{sec:examples}).


\section{Basic Concepts}
A \emph{braided vector space} is a pair $(V, \beta)$ consisting of a vector space $V$ and a braiding $\beta \in \Aut(V \ot V)$, i.e.\ a linear automorphism on $V \ot V$, which satisfies the braid equation
\begin{align*}
	(\beta \ot \id) \circ (\id \ot \beta) \circ (\beta \ot \id) = (\id \ot \beta) \circ (\beta \ot \id) \circ (\id \ot \beta).
\end{align*}
For reasons of clarity and comprehensibility, in the following we use well known braid diagramms (see \cite{Majid}) to express coherences with braidings. In this notation the braid equation can be visualized by
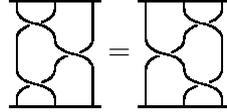
\[ \xy
(-1,14)*{}; (11,14)*{} **\dir{-};
(-1,0)*{}; (11,0)*{} **\dir{-};
(16,14)*{}; (28,14)*{} **\dir{-};
(16,0)*{}; (28,0)*{} **\dir{-};
(0,13)*{}; (0,14)*{} **\dir{-};
(5,13)*{}; (5,14)*{} **\dir{-};
(0,9)*{}; (0,5)*{} **\dir{-};
(10,9)*{}; (10,14)*{} **\dir{-};
(10,5)*{}; (10,0)*{} **\dir{-};
(0,1)*{}; (0,0)*{} **\dir{-};
(5,1)*{}; (5,0)*{} **\dir{-};
\vtwist~{(0,9)}{(5,9)}{(0,13)}{(5,13)};
\vtwist~{(5,5)}{(10,5)}{(5,9)}{(10,9)};
\vtwist~{(0,1)}{(5,1)}{(0,5)}{(5,5)};
(13.5,7)*{=};
(22,13)*{}; (22,14)*{} **\dir{-};
(27,13)*{}; (27,14)*{} **\dir{-};
\vtwist~{(22,9)}{(27,9)}{(22,13)}{(27,13)};
\vtwist~{(17,5)}{(22,5)}{(17,9)}{(22,9)};
(17,9)*{}; (17,14)*{} **\dir{-};
\vtwist~{(22,1)}{(27,1)}{(22,5)}{(27,5)};
(17,5)*{}; (17,0)*{} **\dir{-};
(27,9)*{}; (27,5)*{} **\dir{-};
(22,1)*{}; (22,0)*{} **\dir{-};
(27,1)*{}; (27,0)*{} **\dir{-};
\endxy \]
In our case we do not assume that the braiding fulfils the symmetry condition $\beta^2=\id_{V \ot V}$. So we distinguish them by under and over crossing 
\[ \xy
(-4,3)*{\beta =};
(0,6)*{}; (0,5)*{} **\dir{-};
(5,6)*{}; (5,5)*{} **\dir{-};
(0,1)*{}; (0,0)*{} **\dir{-};
(5,1)*{}; (5,0)*{} **\dir{-};
\vcross~{(0,5)}{(5,5)}{(0,1)}{(5,1)};
\endxy
\hspace*{2em}
\xy
(-6,3)*{\beta^{-1} =};
(0,6)*{}; (0,5)*{} **\dir{-};
(5,6)*{}; (5,5)*{} **\dir{-};
(0,1)*{}; (0,0)*{} **\dir{-};
(5,1)*{}; (5,0)*{} **\dir{-};
\vtwist~{(0,5)}{(5,5)}{(0,1)}{(5,1)};
\endxy \]
Other morphisms will be represented as nodes on a string with the corresponding number of input and ouput strings, e.g.\ for the multiplication $\mu: \calA \ot \calA \to \calA$ and the unit $\eins: \ds{K} \to \calA$ on an algebra $\calA$ resp.\ for the comultiplication $\Delta: \calC \to \calC \ot \calC$ and the counit $\delta: \calC \to \ds{K}$ on a coalgebra $\calC$ we use the shorthand
\[ \xy
(-5,2)*{\mu =};
(0,5)*{}; (0,4)*{} **\dir{-};
(5,5)*{}; (5,4)*{} **\dir{-};
(2.5,2)*{}; (2.5,0)*{} **\dir{-};
(0,4)*{};(5,4)*{} **\crv{(0,2)&(2.5,2)&(5,2)}; 
(2.5,3)*{\cdot}
\endxy 
\hspace*{2em}
\xy
(-5,2.5)*{\eins =};
(0,0)*{}; (0,2)*{} **\dir{-};
(0,0)*{}; (0,5)*{} **\dir{};
(0,2.5)*{\circ};
\endxy
\hspace*{4em}
\xy
(-5,2.5)*{\Delta =};
(0,0)*{}; (0,1)*{} **\dir{-};
(5,0)*{}; (5,1)*{} **\dir{-};
(2.5,3)*{}; (2.5,5)*{} **\dir{-};
(0,1)*{};(5,1)*{} **\crv{(0,3)&(2.5,3)&(5,3)}; 
(2.5,2)*{\cdot}
\endxy
\hspace*{2em}
\xy
(-5,2.5)*{\delta =};
(0,5)*{}; (0,3)*{} **\dir{-};
(0,5)*{}; (0,0)*{} **\dir{};
(0,2.25)*{\circ};
\endxy \]
One defines $\beta_{m,n}: V^{\ot m} \ot V^{\ot n} \to V^{\ot n} \ot V^{\ot m}$ for every braiding $\beta \in \Aut(V \ot V)$ inductively by
\begin{align*}
	\beta_{0,0} := \id_\cC, \quad &\beta_{1,m+1} := (\id_V \ot \beta_{1,m}) \circ (\beta \ot \id_{V^{\ot m}}), \\
	\beta_{1,0} = \beta_{0,1} := \id_V, \quad &\beta_{n+1,m} := (\beta_{n,m} \ot \id_V) \circ (\id_{V^{\ot n}} \ot \beta_{1,m}).
\end{align*}
The braiding $\beta_{m,n}$ can be illustrated by the figure
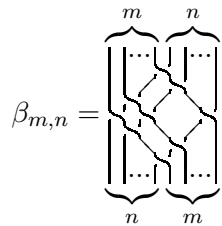
\[ \xy
(-7,9)*{\beta_{m,n} =};
(0,0)*{}; (0,8)*{} **\dir{-};
(2,0)*{}; (2,6)*{} **\dir{-};
(6,0)*{}; (6,2)*{} **\dir{-};
(8,0)*{}; (8,2)*{} **\dir{-};
(10,0)*{}; (10,4)*{} **\dir{-};
(14,0)*{}; (14,8)*{} **\dir{-};
(0,18)*{}; (0,10)*{} **\dir{-};
(2,18)*{}; (2,12)*{} **\dir{-};
(6,18)*{}; (6,16)*{} **\dir{-};
(8,18)*{}; (8,16)*{} **\dir{-};
(10,18)*{}; (10,14)*{} **\dir{-};
(14,18)*{}; (14,10)*{} **\dir{-};
\vcross~{(0,10)}{(2,10)}{(0,8)}{(2,8)};
\vcross~{(2,8)}{(4,8)}{(2,6)}{(4,6)};
\vcross~{(2,12)}{(4,12)}{(2,10)}{(4,10)};
\vcross~{(4,10)}{(6,10)}{(4,8)}{(6,8)};
\vcross~{(6,16)}{(8,16)}{(6,14)}{(8,14)};
\vcross~{(6,4)}{(8,4)}{(6,2)}{(8,2)};
\vcross~{(8,14)}{(10,14)}{(8,12)}{(10,12)};
\vcross~{(8,6)}{(10,6)}{(8,4)}{(10,4)};
\vcross~{(12,10)}{(14,10)}{(12,8)}{(14,8)};
(4,6)*{}; (6,4)*{} **\dir{-};
(6,8)*{}; (8,6)*{} **\dir{-};
(10,12)*{}; (12,10)*{} **\dir{-};
(4,12)*{}; (6,14)*{} **\dir{-};
(6,10)*{}; (8,12)*{} **\dir{-};
(10,6)*{}; (12,8)*{} **\dir{-};
(4,17)*{...};
(12,1)*{...};
(12,17)*{...};
(4,1)*{...};
(3,-3)*{\underbrace{\ }_{n}};
(11,-3)*{\underbrace{\ }_{m}};
(3,21)*{\overbrace{\ }^{m}};
(11,21)*{\overbrace{\ }^{n}};
\endxy 
\]
Note that $(\beta^{-1})_{n,m}$ is its inverse.

Crucial properties for linear maps on braided vector spaces $(V, \beta)$ are the following. A linear map $f: V^{\ot m} \to V^{\ot n}$ is called \emph{$\beta$-invariant}, if
\begin{align*}
	(f \ot \id) \circ \beta_{1,m} = \beta_{1,n} \circ (\id \ot f),
\end{align*}
and accordingly \emph{$\beta^{-1}$-invariant}, if
\begin{align*}
	(\id \ot f) \circ\beta_{m,1} = \beta_{n,1} \circ(f \ot \id).
\end{align*}
In case $f$ fulfils both invariance conditions, we refer to $f$ as \emph{$\beta$-compatible}.

\begin{Rem} \label{kompatibel}
One can easily see that the tensor product and the composition\footnote{If the composition is defined.} of $\beta$-invariant (resp.\ $\beta^{-1}$-invariant and $\beta$-compatible) linear maps is again $\beta$-invariant (resp.\ $\beta^{-1}$-invariant and $\beta$-compatible). For a $\beta$-invariant linear map $f: V^{\ot m} \to V^{\ot n}$ and a $\beta^{-1}$-invariant linear map $g: V^{\ot k} \to V^{\ot l}$, we get
\begin{align*}
	(f \ot g) \circ \beta_{k,m} = \beta_{l,n} \circ (g \ot f).
\end{align*}
As an example for a (trivial) braiding we get the flip-operator $\tau(v \ot w) = w \ot v$. Obviously, all linear maps are $\tau$-compatible.
\end{Rem}

To switch between two braided vector spaces $(V_1, \beta_1)$ and $(V_2, \beta_2)$, we use the notion of a braided morphism. A linear map $f: V_1 \to V_2$ will be called \emph{braided morphism}, if
\begin{align*}
	(f \ot f) \circ \beta_1 = \beta_2 \circ (f \ot f).
\end{align*}
Expressed by braid diagrams, this equation looks like
\[ \xy
(-9,0)*{}; (26,0)*{};
(-1,14)*{}; (6,14)*{} **\dir{-};
(-1,0)*{}; (6,0)*{} **\dir{-};
(11,14)*{}; (18,14)*{} **\dir{-};
(11,0)*{}; (18,0)*{} **\dir{-};
(0,14)*{}; (0,13)*{} **\dir{-};
(5,14)*{}; (5,13)*{} **\dir{-};
\vcross~{(0,13)}{(5,13)}{(0,9)}{(5,9)};
(0,9)*{}; (0,0)*{} **\dir{-};
(5,9)*{}; (5,0)*{} **\dir{-};
(0,4.5)*{\colorbox{white}{f}};
(5,4.5)*{\colorbox{white}{f}};
(8.5,7)*{=};
(0,1)*{}; (0,0)*{} **\dir{-};
(5,1)*{}; (5,0)*{} **\dir{-};
(12,14)*{}; (12,5)*{} **\dir{-};
(17,14)*{}; (17,5)*{} **\dir{-};
\vcross~{(12,5)}{(17,5)}{(12,1)}{(17,1)};
(12,1)*{}; (12,0)*{} **\dir{-};
(17,1)*{}; (17,0)*{} **\dir{-};
(12,9.5)*{\colorbox{white}{f}};
(17,9.5)*{\colorbox{white}{f}};
\endxy \]
A \emph{braided algebra} $(\calA, \mu, \eins, \beta)$ is a unital associative algebra\footnote{Note that in the following all algebra homomorphisms are unital.} $(\calA, \mu, \eins)$ and a braided vector space $(\calA, \beta)$, s.t.\ $\mu$ and $\eins$ are $\beta$-compatible, i.e.\
\begin{align*}
	(\mu \ot \id) \circ \beta_{1,2} = \beta \circ (\id \ot \mu), \quad &(\id \ot \mu) \circ \beta_{2,1} = \beta \circ(\mu \ot \id), \\
	(\eins \ot \id) = \beta \circ (\id \ot \eins), \quad &(\id \ot \eins) = \beta \circ (\eins \ot \id).
\end{align*}
We can visualize these four conditions by
\[
\mubeta
\mubetaa
\einsbeta
\einsbetaa
\]
The multiplication $\mu^{(n)}: \calA^{\ot n} \to \calA$ of $n$ factors is denoted by $\mu^{(n)}(a_1 \ot \cdots \ot a_n) := a_1 \cdots a_n$.

Furthermore, we define $M_1 := \mu$ and $M_n$ for $n\geq2$ inductively  via
\begin{align*}
	M_n := (\mu \ot M_{n-1}) \circ \bigl( \id \ot \beta_{n-1,1} \ot \id^{\ot (n-1)} \bigr).
\end{align*}
 Then $\bigl( \calA^{\ot n}, M_n, \eins^{\ot n}, \beta_{n,n} \bigr)$ becomes a braided algebra.

In particular, 
\begin{align*}
	\bigl( \calA\ot\calA, M, \eins\ot\eins, \beta_{2,2} \bigr)
\end{align*}
is a braided algebra, whereby $M:=M_2=(\mu\ot\mu)\circ(\id\ot\beta\ot\id)$. Note that the usual multiplication map on $\calA \ot \calA$ has changed: We get the braiding $\beta$ instead of the flip operator.

Dually a \emph{braided coalgebra} $(\calC, \Delta, \delta, \beta)$ is a coalgebra $(\calC, \Delta, \delta)$ and a braided vector space $(\calC, \beta)$, s.t.\ $\Delta$ and $\delta$ are $\beta$-compatible, i.e.\
\begin{align*}
	(\Delta \ot \id) \circ \beta = \beta_{1,2} \circ(\id \ot \Delta), \quad &(\id \ot \Delta) \circ\beta = \beta_{2,1} \circ(\Delta \ot \id), \\
	(\delta \ot \id) \circ\beta = (\id \ot \delta), \quad &(\id \ot \delta) \circ\beta = (\delta \ot \id).
\end{align*}
The corresponding diagrams are:
\vspace{-1em}
\[
\rotatebox{180}{\mubeta}
\rotatebox{180}{\mubetaa}
\rotatebox{180}{\einsbeta}
\rotatebox{180}{\einsbetaa}
\]
The comultiplication $\Delta^{(n)}: \calC \to \calC^{\ot n}$ into $n$ factors is $\Delta^{(n)} (c) := \bigl( c_{(1)} \ot \cdots \ot c_{(n)} \bigr)$.

Analogously to the multiplication map in the algebra case, we define 
 $\Lambda_1 := \Delta$ and for $n\geq2$ inductively
\begin{align*}
	\Lambda_n := \bigl( \id \ot \beta_{1,n-1} \ot \id^{\ot (n-1)} \bigr) \circ (\Delta \ot \Lambda_{n-1}).
\end{align*}
Then $\bigl( \calC^{\ot n}, \Lambda_n, \delta^{\ot n}, \beta_{n,n} \bigr)$ becomes a braided coalgebra.
In particular, 
\begin{align*}
	\bigl( \calC\ot\calC, \Lambda, \delta\ot\delta, \beta_{2,2} \bigr)
\end{align*}
is a braided coalgebra, whereby $\Lambda:=\Lambda_2=(\id\ot\beta\ot\id)\circ(\Delta\ot\Delta)$.

\begin{Rem}\label{cocommAlg}
Given a braided algebra $\calA$ and a braided coalgebra $\calC$, it can be easily shown that the \textit{opposite algebra} $\calA^{\op}:= (\calA, \mu \circ \beta, \eins, \beta)$ is a braided algebra and the \textit{coopposite coalgebra} $\calC^{\cop}:= (\calC, \beta \circ \Delta, \delta, \beta)$ is also a braided coalgebra. The algebra $\calA$ is said to be \textit{commutative}, if $\mu = \mu \circ \beta$ and the coalgebra $\calC$ is referred to as \textit{cocommutative} in case $\beta \circ \Delta = \Delta$.
\end{Rem}

A \emph{braided bialgebra} $(\calB, \Delta, \delta, \mu, \eins, \beta)$ is a braided algebra $(\calB, \mu, \eins, \beta)$ and a braided coalgebra $(\calB, \Delta, \delta, \beta)$, s.t.\ $\delta$, $\Delta$ are braided algebra homomorphisms, i.e.\
\begin{align}\label{bialgcond}
	\Delta\circ\mu &= (\mu\ot\mu)\circ(\id\ot\beta\ot\id)\circ(\Delta\ot\Delta)\\
	\delta \circ \mu &= \delta \ot \delta\nonumber
\end{align}
is fulfilled. Equation (\ref{bialgcond}) is called \emph{braided bialgebra condition} and differs from the usual bialgebra condition. In the used graphical calculus the picture
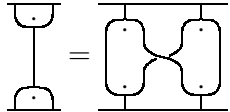
\[ \xy
(-1,14)*{}; (6,14)*{} **\dir{-};
(-1,0)*{}; (6,0)*{} **\dir{-};
(11,14)*{}; (28,14)*{} **\dir{-};
(11,0)*{}; (28,0)*{} **\dir{-};
(0,1)*{};(5,1)*{} **\crv{(0,3)&(2.5,3)&(5,3)}; 
(0,0)*{}; (0,1)*{} **\dir{-};
(5,0)*{}; (5,1)*{} **\dir{-};
(2.5,3)*{}; (2.5,11)*{} **\dir{-};
(2.5,1.5)*{\cdot};
(0,13)*{};(5,13)*{} **\crv{(0,11)&(2.5,11)&(5,11)}; 
(2.5,12)*{\cdot};
(0,13)*{}; (0,14)*{} **\dir{-};
(5,13)*{}; (5,14)*{} **\dir{-};
(8.5,7)*{=};
(14.5,0)*{}; (14.5,2)*{} **\dir{-};
(12,4)*{};(17,4)*{} **\crv{(12,2)&(14.5,2)&(17,2)}; 
(12,4)*{}; (12,5)*{} **\dir{-};
(17,4)*{}; (17,5)*{} **\dir{-};
(14.5,3)*{\cdot};
(24.5,0)*{}; (24.5,2)*{} **\dir{-};
(22,4)*{};(27,4)*{} **\crv{(22,2)&(24.5,2)&(27,2)}; 
(24.5,3)*{\cdot};
\vcross~{(17,9)}{(22,9)}{(17,5)}{(22,5)};
(12,4)*{}; (12,8)*{} **\dir{-};
(27,4)*{}; (27,8)*{} **\dir{-};
(22,5)*{}; (22,4)*{} **\dir{-};
(12,10)*{};(17,10)*{} **\crv{(12,12)&(14.5,12)&(17,12)}; 
(14.5,12)*{}; (14.5,14)*{} **\dir{-};
(14.5,10.5)*{\cdot};
(17,10)*{}; (17,9)*{} **\dir{-};
(12,10)*{}; (12,8)*{} **\dir{-};
(22,10)*{};(27,10)*{} **\crv{(22,12)&(24.5,12)&(27,12)}; 
(24.5,12)*{}; (24.5,14)*{} **\dir{-};
(24.5,10.5)*{\cdot};
(22,10)*{}; (22,9)*{} **\dir{-};
(27,10)*{}; (27,8)*{} **\dir{-};
\endxy \]
represents this equation.

A homomorphism $f: (\calB_1, \beta_1) \to (\calB_2, \beta_2)$ of braided bialgebras is defined as a homomorphism of algebras and coalgebras s.t.\ $(f \ot f) \circ \beta_1 = \beta_2 \circ(f \ot f)$.  

Let $(\calC,\Delta,\delta)$ be a coalgebra and let $(\calA,\mu,\eins)$ be an algebra. Then for linear maps $f,g:\calC\to\calA$ the 
\emph{convolution} is declared by
\begin{align}\label{convol}
f \star g := \mu \circ ( f \ot g ) \circ \Delta.
\end{align}
This product gives the vector space of all linear maps from $\calC$ to $\calA$ the structure of a unital algebra with unit $\eins\delta$. If $\mu, \Delta, f$ and $g$ are $\beta$-invariant, this also holds for $f\star g$, which follows directly from Remark \ref{cocommAlg}.

If $\calB$ is a braided bialgebra, $\calB^{\ot n}$ is a coalgebra and $\calB^{\ot m}$ is an algebra for all $n,m\in\ds{N}$, so convolution is defined for maps $R,T:\calB^{\ot n}\to\calB^{\ot m}$. 

Of course $\ds{C}$ is an algebra, so there is also convolution for linear functionals on a coalgebra. Now we get the formula 
\begin{align*}
\varphi \star \psi := ( \varphi \ot \psi ) \circ \Delta
\end{align*}
for $\varphi,\psi:\calC\to\ds{C}$, since $\ds{C}\ot\ds{C}$ can be identified with $\ds{C}$. For every linear functional $\psi$ on a coalgebra $\calC$ and every $c\in \calC$ the convolution exponential
\begin{align*}
e_\star^{t\psi}(c):=\sum_{n=0}^{\infty} \frac{\psi^{\star n}(c)}{n!}
\end{align*}
converges as a consequence of the \emph{fundamental theorem for coalgebras} (see \cite{Sweedler} and \cite{schurman}). Writing $\varphi_t:=e_\star^{t\psi}$ for $t\geq0$, the $\varphi_t$ constitute a pointwise continuous convolution semigroup, i.e.
\begin{align*}
\varphi_t \star \varphi_s = \varphi_{t+s} \quad\text{and}\quad \lim_{t\rightarrow 0} \varphi_t(c)=\varphi_0(c)=\delta(c)
\end{align*}
for all $c\in\calC$. On the other hand every pointwise continuous convolution semigroup $(\varphi_t)_{t\geq0}$ has a generator $\psi$ s.t.\ $\varphi_t:=e_\star^{t\psi}$ and $\psi$ can be recovered from the convolution semigroup via
\begin{align*}
\psi(c):=\left.\frac{\rd}{\rd t} \varphi_t(c)\right|_{t=0} \equiv \lim_{t \to 0^+} \frac{1}{t} \left.( \varphi_t - \delta \right)(c)
\end{align*}
which is defined for every $c\in\calC$.

A braided bialgebra $\calH$ is called \emph{braided Hopf algebra}, if the identity map is invertible w.r.t.\ the convolution (\ref{convol}), i.e.\ there exists a linear map $S: \calH \to \calH$ called \emph{antipode} with
\begin{align*}
	\mu \circ (S \ot \id) \circ \Delta = \eins \delta = \mu \circ (\id \ot S) \circ \Delta,
\end{align*}
expressed by
\[ \xy
(-1,0)*{}; (6,0)*{} **\dir{-};
(-1,14)*{}; (6,14)*{} **\dir{-};
(2.5,0)*{}; (2.5,2)*{} **\dir{-};
(0,4)*{};(5,4)*{} **\crv{(0,2)&(2.5,2)&(5,2)};
(2.5,3)*{\cdot};
(0,10)*{};(5,10)*{} **\crv{(0,12)&(2.5,12)&(5,12)};
(2.5,12)*{}; (2.5,14)*{} **\dir{-};
(2.5,10.5)*{\cdot};
(0,4)*{}; (0,10)*{} **\dir{-};
(5,4)*{}; (5,10)*{} **\dir{-};
(0,7)*{\colorbox{white}{S}};
(8.5,7)*{=};
(11,0)*{}; (13,0)*{} **\dir{-};
(11,14)*{}; (13,14)*{} **\dir{-};
(12,0)*{}; (12,4)*{} **\dir{-};
(12,10)*{}; (12,14)*{} **\dir{-};
(12,4.5)*{\circ};
(12,9.25)*{\circ};
(15.5,7)*{=};
(18,0)*{}; (25,0)*{} **\dir{-};
(18,14)*{}; (25,14)*{} **\dir{-};
(21.5,0)*{}; (21.5,2)*{} **\dir{-};
(19,4)*{};(24,4)*{} **\crv{(19,2)&(21.5,2)&(24,2)};
(21.5,3)*{\cdot};
(19,10)*{};(24,10)*{} **\crv{(19,12)&(21.5,12)&(24,12)};
(21.5,12)*{}; (21.5,14)*{} **\dir{-};
(21.5,10.5)*{\cdot};
(19,4)*{}; (19,10)*{} **\dir{-};
(24,4)*{}; (24,10)*{} **\dir{-};
(24,7)*{\colorbox{white}{S}};
\endxy \]
\begin{Rem}
For a braided Hopf algebra $(\calH, \Delta, \delta, \mu, \eins, S, \beta)$ the following properties are fulfilled.
\begin{itemize}
\item The antipode $S$ is unique and a braided algebra and colagebra anti-homomorphism,
i.e.\
\begin{align*}
	S \circ \mu = \mu \circ \beta \circ (S \ot S), \quad S \circ \eins = \eins, \quad \quad \Delta \circ S = \beta \circ (S \ot S) \circ \Delta, \quad \delta \circ S = \delta
\end{align*}
are satisfied.\footnote{The first equation follows from the fact that both sides of this equation are convolution inverses of $\mu$. The third equation follows analogously to the first. The second and fourth equation can be shown as in the non-braided case.}
\item The antipode $S$ is $\beta$-compatible, i.e.\
\begin{align*}
	(S \ot \id) \circ \beta = \beta \circ (\id \ot S), \quad (\id \ot S) \circ \beta = \beta \circ (S \ot \id).
\end{align*}
This can be visualized by
\[ \xy
(-9,0)*{}; (26,0)*{};
(-1,14)*{}; (6,14)*{} **\dir{-};
(-1,0)*{}; (6,0)*{} **\dir{-};
(11,14)*{}; (18,14)*{} **\dir{-};
(11,0)*{}; (18,0)*{} **\dir{-};
(0,14)*{}; (0,13)*{} **\dir{-};
(5,14)*{}; (5,13)*{} **\dir{-};
\vcross~{(0,13)}{(5,13)}{(0,9)}{(5,9)};
(0,9)*{}; (0,0)*{} **\dir{-};
(5,9)*{}; (5,0)*{} **\dir{-};
(0,4.5)*{\colorbox{white}{S}};
(8.5,7)*{=};
(0,1)*{}; (0,0)*{} **\dir{-};
(5,1)*{}; (5,0)*{} **\dir{-};
(12,14)*{}; (12,5)*{} **\dir{-};
(17,14)*{}; (17,5)*{} **\dir{-};
\vcross~{(12,5)}{(17,5)}{(12,1)}{(17,1)};
(12,1)*{}; (12,0)*{} **\dir{-};
(17,1)*{}; (17,0)*{} **\dir{-};
(17,9.5)*{\colorbox{white}{S}};
\endxy
%
\xy
(-9,0)*{}; (26,0)*{};
(-1,14)*{}; (6,14)*{} **\dir{-};
(-1,0)*{}; (6,0)*{} **\dir{-};
(11,14)*{}; (18,14)*{} **\dir{-};
(11,0)*{}; (18,0)*{} **\dir{-};
(0,14)*{}; (0,13)*{} **\dir{-};
(5,14)*{}; (5,13)*{} **\dir{-};
\vcross~{(0,13)}{(5,13)}{(0,9)}{(5,9)};
(0,9)*{}; (0,0)*{} **\dir{-};
(5,9)*{}; (5,0)*{} **\dir{-};
(5,4.5)*{\colorbox{white}{S}};
(8.5,7)*{=};
(0,1)*{}; (0,0)*{} **\dir{-};
(5,1)*{}; (5,0)*{} **\dir{-};
(12,14)*{}; (12,5)*{} **\dir{-};
(17,14)*{}; (17,5)*{} **\dir{-};
\vcross~{(12,5)}{(17,5)}{(12,1)}{(17,1)};
(12,1)*{}; (12,0)*{} **\dir{-};
(17,1)*{}; (17,0)*{} **\dir{-};
(12,9.5)*{\colorbox{white}{S}};
\endxy
\]
\item Suppose in addition that $\calH$ is commutative as an algebra or cocommutative as a coalgebra, then $S^2 = \id$ holds. This fact can be seen out of the following diagrams.\footnote{Here we use the cocommutativity $\beta \circ \Delta = \Delta$. The case of commutativity $\mu \circ \beta = \mu$ can be seen analogously.}
\[
\xy
(-1,28)*{}; (1,28)*{} **\dir{-};
(-1,0)*{}; (1,0)*{} **\dir{-};
(0,0)*{}; (0,28)*{} **\dir{-};
\endxy
\hspace*{0.35em}
\xy
(10,14)*{=};
(10,0)*{};
\endxy
\hspace*{0.35em}
%
\xy
(-1,28)*{}; (6,28)*{} **\dir{-};
(-1,0)*{}; (6,0)*{} **\dir{-};
(2.5,28)*{}; (2.5,26)*{} **\dir{-};
(0,24)*{}; (0,22)*{} **\dir{-};
(2.5,0)*{}; (2.5,2)*{} **\dir{-};
(5,24)*{}; (5,4)*{} **\dir{-};
(0,4)*{}; (0,6)*{} **\dir{-};
(0,24)*{};(5,24)*{} **\crv{(0,26)&(2.5,26)&(5,26)};
(0,4)*{};(5,4)*{} **\crv{(0,2)&(2.5,2)&(5,2)};
(0,21.25)*{\circ};
(0,6.5)*{\circ};
(2.5,3)*{\cdot};
(2.5,24.5)*{\cdot};
\endxy
\hspace*{0.35em}
\xy
(10,14)*{=};
(10,0)*{};
\endxy
\hspace*{0.35em}
%
\xy
(-1,28)*{}; (6,28)*{} **\dir{-};
(-1,0)*{}; (6,0)*{} **\dir{-};
(2.5,28)*{}; (2.5,26)*{} **\dir{-};
(0,24)*{}; (0,22)*{} **\dir{-};
(2.5,0)*{}; (2.5,2)*{} **\dir{-};
(5,24)*{}; (5,4)*{} **\dir{-};
(0,4)*{}; (0,11)*{} **\dir{-};
(0,24)*{};(5,24)*{} **\crv{(0,26)&(2.5,26)&(5,26)};
(0,4)*{};(5,4)*{} **\crv{(0,2)&(2.5,2)&(5,2)};
(0,21.25)*{\circ};
(0,11.5)*{\circ};
(2.5,3)*{\cdot};
(2.5,24.5)*{\cdot};
(0,7)*{\colorbox{white}{$S$}};
\endxy
\hspace*{0.35em}
\xy
(10,14)*{=};
(10,0)*{};
\endxy
\hspace*{0.35em}
%
\xy
(-3.5,28)*{}; (6,28)*{} **\dir{-};
(-3.5,0)*{}; (6,0)*{} **\dir{-};
(2.5,28)*{}; (2.5,26)*{} **\dir{-};
(-2.5,13)*{}; (-2.5,22)*{} **\dir{-};
(2.5,13)*{}; (2.5,22)*{} **\dir{-};
(2.5,0)*{}; (2.5,2)*{} **\dir{-};
(5,24)*{}; (5,4)*{} **\dir{-};
(0,4)*{}; (0,11)*{} **\dir{-};
(0,24)*{};(5,24)*{} **\crv{(0,26)&(2.5,26)&(5,26)};
(0,4)*{};(5,4)*{} **\crv{(0,2)&(2.5,2)&(5,2)};
(-2.5,22)*{};(2.5,22)*{} **\crv{(-2.5,24)&(0,24)&(2.5,24)};
(-2.5,13)*{};(2.5,13)*{} **\crv{(-2.5,11)&(0,11)&(2.5,11)};
(2.5,3)*{\cdot};
(2.5,24.5)*{\cdot};
(0,22.5)*{\cdot};
(0,12)*{\cdot};
(0,7)*{\colorbox{white}{$S$}};
(2.5,17.5)*{\colorbox{white}{$S$}};
\endxy
\hspace*{0.35em}
\xy
(10,14)*{=};
(10,0)*{};
\endxy
\hspace*{0.35em}
%
\xy
(-3.5,28)*{}; (6,28)*{} **\dir{-};
(-3.5,0)*{}; (6,0)*{} **\dir{-};
(2.5,28)*{}; (2.5,26)*{} **\dir{-};
(-2.5,16)*{}; (-2.5,22)*{} **\dir{-};
(2.5,16)*{}; (2.5,22)*{} **\dir{-};
(2.5,0)*{}; (2.5,2)*{} **\dir{-};
(5,24)*{}; (5,4)*{} **\dir{-};
(-2.5,6)*{}; (-2.5,12)*{} **\dir{-};
(2.5,6)*{}; (2.5,12)*{} **\dir{-};
(0,24)*{};(5,24)*{} **\crv{(0,26)&(2.5,26)&(5,26)};
(0,4)*{};(5,4)*{} **\crv{(0,2)&(2.5,2)&(5,2)};
(-2.5,22)*{};(2.5,22)*{} **\crv{(-2.5,24)&(0,24)&(2.5,24)};
(-2.5,6)*{};(2.5,6)*{} **\crv{(-2.5,4)&(0,4)&(2.5,4)};
\vcross~{(-2.5,16)}{(2.5,16)}{(-2.5,12)}{(2.5,12)};
(2.5,3)*{\cdot};
(2.5,24.5)*{\cdot};
(0,22.5)*{\cdot};
(0,5)*{\cdot};
(-2.5,9)*{\colorbox{white}{$S$}};
(2.5,9)*{\colorbox{white}{$S$}};
(2.5,19)*{\colorbox{white}{$S$}};
\endxy
\hspace*{0.35em}
\xy
(10,14)*{=};
(10,0)*{};
\endxy
\hspace*{0.35em}
%
\xy
(-1,30)*{}; (8.5,30)*{} **\dir{-};
(5,30)*{}; (5,28)*{} **\dir{-};
(0,20)*{}; (0,8)*{} **\dir{-};
(5,20)*{}; (5,8)*{} **\dir{-};
(5,4)*{}; (5,2)*{} **\dir{-};
(7.5,26)*{}; (7.5,6)*{} **\dir{-};
(-1,2)*{}; (8.5,2)*{} **\dir{-};
\vcross~{(0,24)}{(5,24)}{(0,20)}{(5,20)};
(2.5,26)*{};(7.5,26)*{} **\crv{(2.5,28)&(5,28)&(7.5,28)};
(0,24)*{};(5,24)*{} **\crv{(0,26)&(2.5,26)&(5,26)};
(0,8)*{};(5,8)*{} **\crv{(0,6)&(2.5,6)&(5,6)};
(2.5,6)*{};(7.5,6)*{} **\crv{(2.5,4)&(5,4)&(7.5,4)};
(0,17)*{\colorbox{white}{$S$}};
(0,11)*{\colorbox{white}{$S$}};
(5,14)*{\colorbox{white}{$S$}};
(5,26.5)*{\cdot};
(2.5,24.5)*{\cdot};
(2.5,7)*{\cdot};
(5,5)*{\cdot};
\endxy
\hspace*{0.35em}
\xy
(10,14)*{=};
(10,0)*{};
\endxy
\hspace*{0.35em}
%
\xy
(-1,28)*{}; (8.5,28)*{} **\dir{-};
(-1,0)*{}; (8.5,0)*{} **\dir{-};
(5,28)*{}; (5,26)*{} **\dir{-};
(5,0)*{}; (5,2)*{} **\dir{-};
(7.5,24)*{}; (7.5,4)*{} **\dir{-};
(5,22)*{}; (5,6)*{} **\dir{-};
(0,22)*{}; (0,6)*{} **\dir{-};
(0,22)*{};(5,22)*{} **\crv{(0,24)&(2.5,24)&(5,24)};
(2.5,24)*{};(7.5,24)*{} **\crv{(2.5,26)&(5,26)&(7.5,26)};
(0,6)*{};(5,6)*{} **\crv{(0,4)&(2.5,4)&(5,4)};
(2.5,4)*{};(7.5,4)*{} **\crv{(2.5,2)&(5,2)&(7.5,2)};
(5,24.5)*{\cdot};
(2.5,22.5)*{\cdot};
(5,3)*{\cdot};
(2.5,5)*{\cdot};
(5,14)*{\colorbox{white}{$S$}};
(0,19)*{\colorbox{white}{$S$}};
(0,9)*{\colorbox{white}{$S$}};
\endxy
\hspace*{0.35em}
\xy
(10,14)*{=};
(10,0)*{};
\endxy
\hspace*{0.35em}
%
\xy
(-1,28)*{}; (8.5,28)*{} **\dir{-};
(-1,0)*{}; (8.5,0)*{} **\dir{-};
(2.5,28)*{}; (2.5,26)*{} **\dir{-};
(2.5,0)*{}; (2.5,2)*{} **\dir{-};
(7.5,22)*{}; (7.5,6)*{} **\dir{-};
(2.5,22)*{}; (2.5,6)*{} **\dir{-};
(0,24)*{}; (0,4)*{} **\dir{-};
(0,24)*{};(5,24)*{} **\crv{(0,26)&(2.5,26)&(5,26)};
(2.5,22)*{};(7.5,22)*{} **\crv{(2.5,24)&(5,24)&(7.5,24)};
(0,4)*{};(5,4)*{} **\crv{(0,2)&(2.5,2)&(5,2)};
(2.5,6)*{};(7.5,6)*{} **\crv{(2.5,4)&(5,4)&(7.5,4)};
(2.5,3)*{\cdot};
(5,5)*{\cdot};
(2.5,24.5)*{\cdot};
(5,22.5)*{\cdot};
(2.5,14)*{\colorbox{white}{$S$}};
(0,7)*{\colorbox{white}{$S$}};
(0,21)*{\colorbox{white}{$S$}};
\endxy
\hspace*{0.35em}
\xy
(10,14)*{=};
(10,0)*{};
\endxy
\hspace*{0.35em}
%
\xy
(-1,28)*{}; (1,28)*{} **\dir{-};
(-1,0)*{}; (1,0)*{} **\dir{-};
(0,0)*{}; (0,28)*{} **\dir{-};
(0,9)*{\colorbox{white}{$S$}};
(0,19)*{\colorbox{white}{$S$}};
\endxy
\]
\item It can be easily seen that the braiding $\beta$ is determined by the formula 
\begin{align*}
	\beta = (\mu \ot \mu) \circ \bigl( S \ot (\Delta \circ \mu) \ot S \bigr) \circ (\Delta \ot \Delta).
\end{align*}
\end{itemize}
\end{Rem}

Now we want to define involutions on braided algebraic structures. We follow the definition of an involutive braided bialgebra given by \textsc{U. Franz} and \textsc{R. Schott} (see \cite{FSS} or \cite{Franz}, section 3.8), which differs from that in \cite{Majid}.

A \emph{braided $*$-bialgebra} $(\calB, \Delta, \delta, \mu, \eins, \beta, *)$ is a braided bialgebra $(\calB, \Delta, \delta, \mu, \eins, \beta)$ with an anti-linear map $*: \calB \to \calB$, s.t.\ $(\calB, \mu, \eins, *)$ is a $*$-algebra and $\calB \ot \calB$ respects the involution in such a way, that the canonical embeddings $\calB \rightarrow \calB\ot\calB \leftarrow \calB$ and $\Delta$ are $*$-algebra homomorphisms.

\begin{Rem} From the definition above we get the following properties concerning braided $*$-bialgebras: 
\begin{itemize}
	\item If the canonical embeddings $a \mapsto a \ot \eins$ and $a \mapsto \eins \ot a$ are $*$-algebra homomorphisms, we have $(a \ot \eins)^* = a^* \ot \eins$ and $(\eins \ot a)^* = \eins \ot a^*$. Since the involution on $\calB \ot \calB$ shall be an anti-algebra homomorphism we get
\begin{align*}
	(a \ot b)^{*} &= \bigl((a \ot \eins)(\eins \ot b)\bigr)^* = (\eins \ot b)^*(a \ot \eins)^* = (\eins \ot b^*)(a^* \ot \eins) \\
	&= (\mu \ot \mu) \circ (\id \ot \beta \ot \id) (\eins \ot b^* \ot a^* \ot \eins)\\
	&= \beta (b^* \ot a^*) = \beta \circ (* \ot *) \circ \tau (a \ot b).
\end{align*}
(Note that we used the $\beta$-compatible multiplication $M = (\mu \ot \mu) \circ (\id \ot \beta \ot \id)$ on $\calB \ot \calB$ instead of the usual one.) It follows that the involution on $\calB \ot \calB$ is given by
\begin{align*}
	\beta \circ (* \ot *) \circ \tau =: *_{\calB \ot \calB},
\end{align*}
where $\tau$ is the usual flip operator $\tau(a \ot b) = b \ot a$.
\item Summarizing, we get an equivalent definition for braided $*$-bialgebras:
A braided $*$-bialgebra $\calB$ is a braided bialgebra $\calB$ with an involution $*$, s.t.\ 
\begin{align*}
	(*_{\calB \ot \calB})^2 =\bigl( \beta \circ (* \ot *) \circ \tau \bigr)^2 = \id_{\calB \ot \calB}.
\end{align*}
The involution $*$ is, in general, not $\beta$-compatible, but fulfils
\begin{align*}
	\beta \circ (* \ot *) \circ \tau = (* \ot *) \circ \tau \circ \beta^{-1}.
\end{align*}
This condition contains the flip operator and the braiding. To avoid confusion, we did not use braid diagrams in calculation with $*$.

\item Note that the multiplication $\mu$ and the comultiplication $\Delta$ fulfil
\begin{align*}
	* \circ \mu = \mu \circ (* \ot *) \circ \tau, \quad \Delta \circ * = *_{\calB \ot \calB} \circ \Delta.
\end{align*}
\end{itemize}
\end{Rem}

Now we want to show a central but not obvious property of hermitian, bilinear functionals on a braided coalgebra. We call a bilinear functional $K: \calC \ot \calC \to \cC$ on a braided coalgebra $\calC$ hermitian, if $K (a^* \ot b^*) = \overline{K (b \ot a)}$ for all $a, b \in \calC$. Note that this condition differs from $K((a\ot b)^{*_{\calB\ot\calB}})=\overline{K(a\ot b)}$.

\begin{Prop} \label{Lhermit}
Suppose we have two hermitian, linear functionals $K, L$ on $\calC \ot \calC$ on a braided $*$-coalgebra $(\calC, \Delta, \delta, \beta)$. If $K$ is $\beta$-invariant or $L$ is $\beta^{-1}$-invariant, the convolution $K \star L$ is hermitian, too.
\end{Prop}

\begin{proof}
That $K$ and $L$ are hermitian means $K = \overline{K} \circ (* \ot *) \circ \tau$ and $L = \overline{L} \circ (* \ot *) \circ \tau$.
\begin{align*}
	&(\overline{K \star L}) \circ (* \ot *) \circ \tau \\
	& \quad = (\overline{K} \ot \overline{L}) \circ \Lambda \circ (* \ot *) \circ \tau \displaybreak[0]\\
	& \quad = (K \ot L) \circ (*\ot*\ot*\ot*) \circ (\tau\ot\tau) \circ (\id\ot\beta\ot\id) \circ (\Delta\ot\Delta) \circ (* \ot *) \circ \tau \displaybreak[0]\\
	& \quad = (K \ot L) \circ (*\ot*\ot*\ot*) \circ (\tau\ot\tau) \circ (\id\ot\beta\ot\id) \circ (\beta\ot\beta) \circ (*\ot*\ot*\ot*) \\
	&\qquad \qquad \circ (\tau\ot\tau) \circ (\Delta\ot\Delta) \circ \tau \displaybreak[0]\\
	& \quad = (K \ot L) \circ (*\ot*\ot*\ot*) \circ (\tau\ot\tau) \circ (\id\ot\beta\ot\id) \circ (*\ot*\ot*\ot*) \circ (\tau\ot\tau) \\
	&\qquad \qquad \circ (\beta^{-1}\ot\beta^{-1}) \circ \tau_{2,2} \circ (\Delta\ot\Delta) \displaybreak[0]\\
	& \quad = (K \ot L) \circ (*\ot*\ot*\ot*) \circ (\tau\ot\tau) \circ (\id\ot\beta\ot\id) \circ (*\ot*\ot*\ot*) \circ (\tau\ot\tau) \\
	&\qquad \qquad \circ \tau_{2,2} \circ (\beta^{-1}\ot\beta^{-1}) \circ (\Delta\ot\Delta) \displaybreak[0]\\
	& \quad = (K \ot L) \circ (*\ot*\ot*\ot*) \circ (\tau\ot\tau) \circ (\id\ot\beta\ot\id) \circ (*\ot*\ot*\ot*) \circ (\id\ot\tau\ot\id) \\
	&\qquad \qquad \circ \tau_{(14)} \circ (\beta^{-1}\ot\beta^{-1}) \circ (\Delta\ot\Delta) \displaybreak[0]\\
	& \quad = (K \ot L) \circ (\tau\ot\tau) \circ (*\ot*\ot*\ot*) \circ (*\ot*\ot*\ot*) \circ (\id\ot\tau\ot\id)\\
	&\qquad \qquad \circ (\id\ot\beta^{-1}\ot\id) \circ \tau_{(14)} \circ (\beta^{-1}\ot\beta^{-1}) \circ (\Delta\ot\Delta) \displaybreak[0]\\
	& \quad = (K \ot L) \circ (\tau\ot\tau) \circ (\id\ot\tau\ot\id) \circ \tau_{(14)} \circ (\id\ot\beta^{-1}\ot\id) \circ (\beta^{-1}\ot\beta^{-1}) \circ (\Delta\ot\Delta) \displaybreak[0]\\
	& \quad = (K \ot L) \circ \tau_{2,2} \circ \beta_{2,2}^{-1} \circ \Lambda \\
	& \quad = (L \ot K) \circ \beta_{2,2}^{-1} \circ \Lambda = (K \ot L) \circ \Lambda = K \star L,
\end{align*}
wherein $\tau_{(14)}$ is defined by $a \ot b \ot c \ot d \mapsto d \ot b \ot c \ot a$. We used the $\beta$-invariance of $K$ (resp. $\beta^{-1}$-invariance of $L$) twice at the last step.
\end{proof}

It follows directly from this proposition that, for a hermitian, $\beta$-compatible, linear functional $L: \calB \ot \calB \to \cC$, the convolution exponential $\e_\star^{tL}$ is hermitian for every $t \in \rR$, too. We will need this in the following section.


\section{The Generator of an Additive Deformation}\label{secgenerator}
Let $(\calB, \Delta, \delta, \mu, \eins, \beta)$ be a braided bialgebra. Then we call a family $(\mu_t)_{t \in \rR}$ of $\beta$-compatible maps $\mu_t : \calB \ot \calB \to \calB$ an \emph{additive deformation}, if
\begin{itemize}
	\item $\mu_0 = \mu$,
	\item $\calB_t = (\calB, \mu_t, \eins, \beta)$ is a braided unital algebra for all $t \in \rR$,
	\item $\Delta \circ \mu_{t+s} = (\mu_t \ot \mu_s) \circ (\id \ot \beta \ot \id) \circ (\Delta \ot \Delta)$ for all $t, s \in \rR$,
	\item $\delta \circ \mu_t \xrightarrow[t \to 0]{} \delta \circ \mu_0 = \delta \ot \delta$ pointwise.
\end{itemize}
Assume $\calB$ is a braided $*$-bialgebra. Then we call $(\mu_t)_{t \in \rR}$ an \emph{additive $*$-deformation}, if in addition
\begin{itemize}
	\item $\mu_t(a^* \ot b^*) = \mu_t(b \ot a)^*$ for all $t \in \rR$,
\end{itemize}
i.e.\ $* \circ \mu_t = \mu_t \circ (* \ot *) \circ \tau$.
\begin{Rem}
The third condition states that the comultiplication $\Delta$ is a $*$-algebra homomorphism from $\calB_{t+s}$ into $\calB_t \ot \calB_s$, as the comultiplication on the bialgebra $\calB \ot \calB$ is defined by $\Lambda = (\id \ot \beta \ot \id) \circ (\Delta \ot \Delta)$.
\end{Rem}
\begin{Thm}
Suppose that $\calB$ is a braided bialgebra and $(\mu_t)_{t \in \rR}$ an additive deformation. Then there exists a $\beta$-compatible linear functional $L: \calB \ot \calB \to \cC$ given by
\begin{align*}
	L = \left. \frac{\rd}{\rd t} \delta \circ \mu_t \right|_{t=0} \equiv \lim_{t \to 0^+} \frac{1}{t} \left.(\delta\circ \mu_t - \delta \ot \delta \right)
\end{align*} 
pointwise. Furthermore, $L$ fulfils
\begin{itemize}
	\item[$(i)$] $\mu_t = \mu \star \e_\star^{tL}$,
	\item[$(ii)$] $L \star \mu = \mu \star L$,
	\item[$(iii)$] $L(\eins \ot \eins) = 0$,
	\item[$(iv)$] $\partial L = \delta \ot L - L \circ (\mu \ot \id) + L \circ (\id \ot \mu) - L \ot \delta = 0$.\footnote{Condition $(iv)$ is called cocycle property. The operator $\partial$ is the coboundary operator in the Hochschild cohomology associated to the $\cC$-$\cC$-bimodule structure on $\calB$, defined by $\alpha.b.\beta:=\alpha\beta b$ for $\alpha,\beta\in\cC$ and $b\in\calB$, see \cite{Wirth}.}
\end{itemize}
If $\calB$ is even a braided $*$-bialgebra and $(\mu_t)_{t \in \rR}$ an additive $*$-deformation, then we have additionally
\begin{itemize}
	\item[$(v)$] $\overline{L(a \ot b)} = L (b^* \ot a^*)$.
\end{itemize}

Conversely, suppose that the $\beta$-compatible linear functional $L: \calB \ot \calB \to \cC$ on a braided bialgebra $\calB$ fulfils conditions $(ii)$ to $(iv)$, then $(i)$ defines an additive deformation.
If $\calB$ is a braided $*$-bialgebra and $L$ satisfies additionally $(v)$, then $(i)$ defines an additive $*$-deformation.
\end{Thm}

\begin{proof} (in the non-braided case due to J. \textsc{Wirth}, see \cite{Wirth})
Let $(\mu_t)_{t \in \ds R}$  be an additive deformation. It follows that 
\begin{align*}
	(\delta \circ \mu_t) \star (\delta \circ \mu_s) = (\delta \ot \delta) \circ (\mu_t \ot \mu_s) \circ \Lambda = (\delta \ot \delta) \circ \Delta \circ \mu_{t+s} = \delta \circ \mu_{t+s}.
\end{align*}
Thus $(\delta \circ \mu_t)_{t \in \ds R}$ is a continuous convolution semigroup,\footnote{The continuity is just the last condition in the definition of an additive deformation.} which implies that there exists a generator $L = \lim_{t \to 0} \frac{1}{t}\left.( \mu_t - \delta \ot \delta \right)$ with $\delta \circ \mu_t = \e_{\star}^{tL}$.
Moreover,
\begin{align*}
	\mu \star (\delta \circ \mu_t) = \bigl( \mu \ot (\delta \circ \mu_t) \bigr) \circ \Lambda = (\id \ot \delta) \circ (\mu \ot \mu_t) \circ \Lambda = (\id \ot \delta) \circ \Delta \circ \mu_t = \mu_t.
\end{align*}
Analogously, $(\delta \circ \mu_t) \star \mu = \mu_t$ holds. The differentiation of
\begin{align*}
	\e_\star^{tL} \star \mu = \mu \star \e_\star^{tL}, \quad \mu_t(\eins \ot \eins) = \eins, \quad \mu_t \circ (\mu_t \ot \id) = \mu_t \circ (\id \ot \mu_t)
\end{align*}
and $* \circ \mu_t = \mu_t \circ (* \ot *) \circ \tau$ (in the $*$-case) at $t = 0$ gives the properties $(ii)$ to $(v)$. The $\beta$-compatibility of $\mu_t$ implies the $\beta$-compatibility of $L$.

Conversely, assume $L: \calB \ot \calB \to \ds C$ fulfils $(ii)$ to $(iv)$. We want to show the associativity of $\mu_t:= \mu \star \e_{\star}^{tL}$.\footnote{Obviously, the multiplication $\mu_t$ is $\beta$-compatible due to remark \ref{kompatibel}.} First observe that\footnote{Equation (\ref{bew:asso}) results firstly from the fact that $L \circ (\id \ot \mu) + \delta \ot L = L \circ (\mu \ot \id) + L \ot \delta$ since the terms on the left and on the right side commute under the convolution, secondly from the fact that $(\id \ot \mu)$ and $(\mu \ot \id)$ are coalgebra homomorphisms and finally from the fact that $(\delta \ot K_1) \star (\delta \ot K_2) = \delta \ot (K_1 \star K_2)$.}
\begin{align} \label{bew:asso}
	\bigl(\e_{\star}^{tL} \circ (\id \ot \mu)\bigr) \star (\delta \ot \e_{\star}^{tL}) = \bigl( \e_{\star}^{tL} \circ (\mu \ot \id)\bigr) \star (\e_{\star}^{tL} \ot \delta).
\end{align}
Now we calculate $\mu_t \circ (\id \ot \mu_t)$ with braid diagrams, where $\bullet$ means $\e_\star^{tL}$. We get
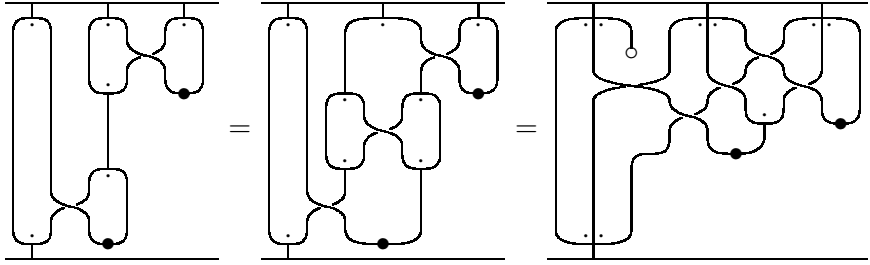
\[
\xy
(-1,0)*{}; (27,0)*{} **\dir{-};
(-1,34)*{}; (27,34)*{} **\dir{-};
(2.5,34)*{}; (2.5,32)*{} **\dir{-};
(12.5,34)*{}; (12.5,32)*{} **\dir{-};
(22.5,34)*{}; (22.5,32)*{} **\dir{-};
(15,30)*{}; (15,29)*{} **\dir{-};
(20,30)*{}; (20,29)*{} **\dir{-};
(15,25)*{}; (15,24)*{} **\dir{-};
(20,25)*{}; (20,24)*{} **\dir{-};
(10,30)*{}; (10,24)*{} **\dir{-};
(25,30)*{}; (25,24)*{} **\dir{-};
(2.5,0)*{}; (2.5,2)*{} **\dir{-};
(0,4)*{}; (0,30)*{} **\dir{-};
(5,4)*{}; (5,5)*{} **\dir{-};
(10,4)*{}; (10,5)*{} **\dir{-};
(5,9)*{}; (5,30)*{} **\dir{-};
(10,9)*{}; (10,10)*{} **\dir{-};
(15,4)*{}; (15,10)*{} **\dir{-};
(12.5,12)*{}; (12.5,22)*{} **\dir{-};
(0,30)*{};(5,30)*{} **\crv{(0,32)&(2.5,32)&(5,32)};
(10,30)*{};(15,30)*{} **\crv{(10,32)&(12.5,32)&(15,32)};
(20,30)*{};(25,30)*{} **\crv{(20,32)&(22.5,32)&(25,32)};
(10,24)*{};(15,24)*{} **\crv{(10,22)&(12.5,22)&(15,22)};
(20,24)*{};(25,24)*{} **\crv{(20,22)&(22.5,22)&(25,22)};
(0,4)*{};(5,4)*{} **\crv{(0,2)&(2.5,2)&(5,2)};
(10,10)*{};(15,10)*{} **\crv{(10,12)&(12.5,12)&(15,12)};
(10,4)*{};(15,4)*{} **\crv{(10,2)&(12.5,2)&(15,2)};
(2.5,31)*{\cdot};
(12.5,31)*{\cdot};
(22.5,31)*{\cdot};
(12.5,23)*{\cdot};
(2.5,3)*{\cdot};
(12.5,11)*{\cdot};
\vcross~{(15,29)}{(20,29)}{(15,25)}{(20,25)};
\vtwist~{(5,5)}{(10,5)}{(5,9)}{(10,9)};
(22.5,21.9)*{\bullet};
(12.5,1.9)*{\bullet};
\endxy
\hspace*{0.35em}
\xy
(10,17)*{=};
(10,0)*{};
\endxy
\hspace*{0.35em}
\xy
(-1,0)*{}; (31,0)*{} **\dir{-};
(-1,34)*{}; (31,34)*{} **\dir{-};
(2.5,0)*{}; (2.5,2)*{} **\dir{-};
(0,4)*{}; (0,30)*{} **\dir{-};
(5,4)*{}; (5,5)*{} **\dir{-};
(5,9)*{}; (5,30)*{} **\dir{-};
(2.5,32)*{}; (2.5,34)*{} **\dir{-};
(10,4)*{}; (10,5)*{} **\dir{-};
(10,9)*{}; (10,12)*{} **\dir{-};
(7.5,14)*{}; (7.5,20)*{} **\dir{-};
(12.5,14)*{}; (12.5,15)*{} **\dir{-};
(12.5,19)*{}; (12.5,20)*{} **\dir{-};
(10,22)*{}; (10,30)*{} **\dir{-};
(15,32)*{}; (15,34)*{} **\dir{-};
(20,4)*{}; (20,12)*{} **\dir{-};
(17.5,14)*{}; (17.5,15)*{} **\dir{-};
(17.5,19)*{}; (17.5,20)*{} **\dir{-};
(22.5,14)*{}; (22.5,20)*{} **\dir{-};
(20,22)*{}; (20,25)*{} **\dir{-};
(20,29)*{}; (20,30)*{} **\dir{-};
(25,24)*{}; (25,25)*{} **\dir{-};
(25,29)*{}; (25,30)*{} **\dir{-};
(27.5,32)*{}; (27.5,34)*{} **\dir{-};
(30,24)*{}; (30,30)*{} **\dir{-};
(0,4)*{};(5,4)*{} **\crv{(0,2)&(2.5,2)&(5,2)};
(0,30)*{};(5,30)*{} **\crv{(0,32)&(2.5,32)&(5,32)};
(7.5,14)*{};(12.5,14)*{} **\crv{(7.5,12)&(10,12)&(12.5,12)};
(7.5,20)*{};(12.5,20)*{} **\crv{(7.5,22)&(10,22)&(12.5,22)};
(10,4)*{}; (20,4)*{} **\crv{(10,2)&(15,2)&(20,2)};
(10,30)*{};(20,30)*{} **\crv{(10,32)&(15,32)&(20,32)};
(17.5,14)*{};(22.5,14)*{} **\crv{(17.5,12)&(20,12)&(22.5,12)};
(17.5,20)*{};(22.5,20)*{} **\crv{(17.5,22)&(20,22)&(22.5,22)};
(25,24)*{};(30,24)*{} **\crv{(25,22)&(27.5,22)&(30,22)};
(25,30)*{};(30,30)*{} **\crv{(25,32)&(27.5,32)&(30,32)};
(2.5,3)*{\cdot};
(2.5,31)*{\cdot};
(10,13)*{\cdot};
(10,21)*{\cdot};
(15,31)*{\cdot};
(20,13)*{\cdot};
(20,21)*{\cdot};
(27.5,31)*{\cdot};
\vtwist~{(5,5)}{(10,5)}{(5,9)}{(10,9)};
\vtwist~{(12.5,15)}{(17.5,15)}{(12.5,19)}{(17.5,19)};
\vtwist~{(20,25)}{(25,25)}{(20,29)}{(25,29)};
(15,1.9)*{\bullet};
(27.5,21.9)*{\bullet};
\endxy
\hspace*{0.35em}
\xy
(10,17)*{=};
(10,0)*{};
\endxy
\hspace*{0.35em}
\xy
(-1,0)*{}; (41,0)*{} **\dir{-};
(-1,34)*{}; (41,34)*{} **\dir{-};
(5,34)*{}; (5,32)*{} **\dir{-};
(5,0)*{}; (5,2)*{} **\dir{-};
(0,4)*{}; (0,30)*{} **\dir{-};
(20,34)*{}; (20,25)*{} **\dir{-};
(35,34)*{}; (35,25)*{} **\dir{-};
(25,30)*{}; (25,29)*{} **\dir{-};
(30,30)*{}; (30,29)*{} **\dir{-};
(10,30)*{}; (10,28)*{} **\dir{-};
(35,21)*{}; (35,20)*{} **\dir{-};
(40,30)*{}; (40,20)*{} **\dir{-};
(30,21)*{}; (30,20)*{} **\dir{-};
(25,21)*{}; (25,20)*{} **\dir{-};
(20,17)*{}; (20,16)*{} **\dir{-};
(27.5,18)*{}; (27.5,16)*{} **\dir{-};
(15,25)*{}; (15,30)*{} **\dir{-};
(5,25)*{}; (5,32)*{} **\dir{-};
(5,2)*{}; (5,21)*{} **\dir{-};
(15,16)*{}; (15,17)*{} **\dir{-};
(10,4)*{}; (10,12)*{} **\dir{-};
(0,30)*{};(10,30)*{} **\crv{(0,32)&(5,32)&(10,32)};
(0,4)*{};(10,4)*{} **\crv{(0,2)&(5,2)&(10,2)};
(15,30)*{};(25,30)*{} **\crv{(15,32)&(20,32)&(25,32)};
(30,30)*{};(40,30)*{} **\crv{(30,32)&(35,32)&(40,32)};
(35,20)*{};(40,20)*{} **\crv{(35,18)&(37.5,18)&(40,18)};
(25,20)*{};(30,20)*{} **\crv{(25,18)&(27.5,18)&(30,18)};
(20,16)*{};(27.5,16)*{} **\crv{(20,14)&(23.25,14)&(27,14)};
(10,12)*{};(15,16)*{} **\crv{(10,14)&(12.5,14)&(15,14)};
\vcross~{(25,29)}{(30,29)}{(25,25)}{(30,25)};
\vcross~{(20,25)}{(25,25)}{(20,21)}{(25,21)};
\vcross~{(30,25)}{(35,25)}{(30,21)}{(35,21)};
\vcross~{(15,21)}{(20,21)}{(15,17)}{(20,17)};
\vtwist~{(5,21)}{(15,21)}{(5,25)}{(15,25)};
(21,31)*{\cdot};
(36,31)*{\cdot};
(27.5,19)*{\cdot};
(6,31)*{\cdot};
(4,3)*{\cdot};
(34,31)*{\cdot};
(19,31)*{\cdot};
(4,31)*{\cdot};
(6,3)*{\cdot};
(10,27.25)*{\circ};
(37.5,17.9)*{\bullet};
(23.75,13.9)*{\bullet};
\endxy
\]
The last diagram is equal to $\mu^{(3)} \star \left( \e_\star^{tL} \circ (\id \ot \mu) \right) \star \left(\delta\ot  \e_\star^{tL} \right)$. In the same manner, we get $\mu_t \circ (\mu_t \ot \id)=\mu^{(3)} \star \left( \e_\star^{tL} \circ (\mu \ot \id) \right) \star \left(\e_\star^{tL} \ot \delta \right)$. Now associativity of $\mu_t$ follows from (\ref{bew:asso}). We also have
\begin{align*}
	\mu_t (\eins \ot \eins) = \mu \star \e_\star^{tL} (\eins \ot \eins) = \e_\star^{tL(\eins \ot \eins)} \mu (\eins \ot \eins) = \eins
\end{align*}
for all $t \in \rR$ and, obviously, $\mu_0 = \mu$ is fulfilled. Now we prove that $\Delta:\calB_{t+s}\to\calB_t\ot\calB_s$ is an algebra homomorphism
\begin{align*}
	(\mu_t \ot \mu_s) \circ \Lambda &= (\e_\star^{tL} \ot \mu \ot \mu \ot \e_\star^{sL}) \circ \Lambda^{(4)} \\
	&= (\e_\star^{tL} \ot \mu \ot \mu \ot \e_\star^{sL}) \circ (\id\ot\id \ot \Lambda \ot \id\ot\id) \circ \Lambda^{(3)} \displaybreak[0]\\
	&= \Delta \circ (\e_\star^{tL} \ot \mu \ot \e_\star^{sL}) \circ \Lambda^{(3)} \displaybreak[0]\\
	&= \Delta \circ (\mu \ot \e_\star^{tL} \ot \e_\star^{sL}) \circ \Lambda^{(3)} \displaybreak[0]\\
	&= \Delta \circ (\mu \ot \e_\star^{(t+s)L}) \circ \Lambda \\
	&= \Delta \circ \mu_{(t+s)}.
\end{align*}
At last we need the implication: If $L$ is hermitian, $\calB_t$ becomes a $*$-algebra, i.e. $\mu_t(a \ot b) = \mu_t(b^* \ot a^*)^*$:
\begin{align*}
	&* \circ \, \mu_t \circ (* \ot *) \circ \tau\\
	&\quad= * \circ (\mu \star \e_\star^{tL}) \circ (* \ot *) \circ \tau \displaybreak[0]\\
	&\quad= * \circ (\mu \ot \e_\star^{tL}) \circ (\id \ot \beta \ot \id) \circ (\Delta \ot \Delta) \circ (* \ot *) \circ \tau \displaybreak[0]\\
	&\quad= (\mu \ot \e_\star^{tL}) \circ (*\ot*\ot*\ot*) \circ (\tau\ot\tau) \circ (\id \ot \beta \ot \id) \circ (\Delta \ot \Delta) \circ (* \ot *) \circ \tau
\end{align*}
This last expression is the same as line three in the proof of proposition \ref{Lhermit} with $K$ replaced by $\mu$ and $L$ replaced by $e_\star^{tL}$. The same manipulations can be performed and we arrive at
\begin{align*}
  (\mu \ot \e_\star^{tL}) \circ \tau_{2,2} \circ \beta_{2,2}^{-1} \circ \Lambda = (\e_\star^{tL} \ot \mu) \circ \beta_{2,2}^{-1} \circ \Lambda = (\mu \ot \e_\star^{tL}) \circ \Lambda = \mu_t,
\end{align*}
using the  $\beta^{-1}$-invariance of $L$.
\end{proof}


\section{Hopf-Deformations}\label{secHopf}

In this section, we want to show the existence of deformed antipodes on braided Hopf ($*$-)algebras and explore their properties. 

Let $\bigl(\calH, \Delta, \delta, \mu, \eins, S, \beta, (*)\bigr)$ be a Hopf ($*$-)algebra.
If we have an additive deformation $(\mu_t)_{t \in \rR}$ with generator $L$, the equation
\begin{align}\label{eq:sigma}
	L \circ (\id \ot S) \circ \Delta = L \circ (S \ot \id) \circ \Delta
\end{align}
holds because of
\begin{align*}
	0 &= \partial L \left( a_{(1)} \ot S(a_{(2)}) \ot a_{(3)} \right) \\
	&= L \left( S(a_{(1)}) \ot a_{(2)} \right) - \underbrace{L(\eins \ot a)}_{=0} + \underbrace{L(a \ot \eins)}_{=0} - L \left( a_{(1)} \ot S(a_{(2)}) \right).
\end{align*}

\begin{Lem}\label{Lem:muK}
Let $K$ be a $\beta$-invariant linear functional on $\calH \ot \calH$ and $\widetilde{K} := K \circ (S \ot \id) \circ \Delta$. Then
\begin{align*}
	K \star \mu = \mu \star K \quad \text{implies} \quad \widetilde{K}\star \id = \id \star \widetilde{K}.
\end{align*}
\end{Lem}

\begin{proof}
First we use $\Delta \circ S = \beta \circ ( S\ot S )\circ \Delta$ in order to get
\begin{align*}
	\Lambda \circ (S\ot \id ) \circ \Delta &=( \id \ot \beta \ot \id )\circ (\Delta \ot \Delta ) \circ ( S\ot \id ) \circ \Delta \\
	&=( \id \ot \beta \ot \id )\circ (\beta \ot \id \ot \id )\circ (S\ot S \ot \id \ot \id )\circ \Delta^{(4)}\\
	&=(\beta_{1,2}\ot \id ) \circ (S\ot S \ot \id \ot \id )\circ \Delta^{(4)}.
\end{align*}
This allows us to calculate
\begin{align*}
	(K\star \mu ) \circ ( S\ot \id ) \circ \Delta &=(K \ot \mu ) \circ (\beta_{1,2}\ot \id ) \circ (S\ot S \ot \id \ot \id )\circ \Delta^{(4)}\\
	&=\mu \circ (S \ot \widetilde{K} \ot \id ) \circ \Delta^{(3)}
\end{align*}
using $\beta$-invariance of $K$. Next we get
\begin{align*}
	(\mu\star K ) \circ ( S\ot \id ) \circ \Delta &=(\mu \ot K ) \circ (\beta_{1,2}\ot \id ) \circ (S\ot S \ot \id \ot \id )\circ \Delta^{(4)}\\
	&=(\id \ot K) \circ (\beta\ot \id ) \circ (\id\ot \mu \ot \id) \circ (S \ot S \ot \id \ot \id ) \circ \Delta^{(4)}\\
	&=(\id \ot K) \circ (\beta\ot \id ) \circ (S \ot \eins \ot \id ) \circ \Delta=\eins \circ \widetilde{K}
\end{align*}
using $\beta$ invariance of $\mu$ and $\eins$ as well as the antipode equation.
Combining these two equations, it follows from $K \star \mu = \mu \star K$ that
\begin{align*}
	\eins \circ \widetilde{K}=(\mu\star K ) \circ ( S\ot \id ) \circ \Delta =( K \star \mu ) \circ ( S\ot \id ) \circ \Delta
	= \mu \circ (S \ot \widetilde{K} \ot \id ) \circ \Delta^{(3)},
\end{align*}
or in Sweedler notation (suppressing the unit)
\begin{align*}
	\widetilde{K}(a) = S(a_{(1)}) \widetilde{K}(a_{(2)}) a_{(3)}.
\end{align*}
Now it follows easily that
\begin{align*}
	(\id \star \widetilde{K} )(a)= a_{(1)} \widetilde{K}(a_{(2)}) =a_{(1)} \bigl(S(a_{(2)}) \widetilde{K}(a_{(3)}) a_{(4)}\bigr)
	= \widetilde{K}(a_{(1)}) a_{(2)}=(\widetilde{K} \star \id )(a)
\end{align*}
for all $a \in \calH$.
\end{proof}

\begin{Cor}
The family $F_t := \e_\star^{tL} \circ (S \ot \id) \circ \Delta$ is a continuous convolution semigroup and
\begin{align*}
	\e_\star^{tL} \circ (S \ot \id) \circ \Delta = \e_\star^{tL} \circ (\id \ot S) \circ \Delta = \e_\star^{t \sigma}
\end{align*}
with $\sigma := L \circ (S \ot \id) \circ \Delta$.
\end{Cor}

\begin{proof}
The continuous convolution semigroup $e_\star^{tL}$ fulfils $\e_\star^{tL} \star \mu = \mu \star \e_\star^{tL}$. Because of lemma \ref{Lem:muK}, we have $F_t \star \id = \id \star F_t$, so we get
\begin{align*}
	F_t \star F_s &= (F_t \ot F_s) \circ \Delta = \e_\star^{tL} \circ (S \ot \id \ot F_s) \circ \Delta^{(3)} \\
	&= \e_\star^{tL} \circ (S \ot F_s \ot \id) \circ \Delta^{(3)} \displaybreak[0]\\
	&= \e_\star^{tL} \circ (\id \ot \e_\star^{sL} \ot \id) \circ (S \ot S \ot \id \ot \id) \circ \Delta^{(4)} \displaybreak[0]\\
	&= \e_\star^{tL} \circ (\id \ot \e_\star^{sL} \ot \id) \circ (\beta^{-1} \ot \id \ot \id) \circ (\Delta \ot \Delta) \circ (S \ot \id) \circ \Delta \displaybreak[0]\\
	&= \e_\star^{tL} \circ (\e_\star^{sL} \ot \id \ot \id) \circ (\beta_{1,2} \ot \id) \circ (\beta^{-1} \ot \id \ot \id) \circ (\Delta \ot \Delta) \circ (S \ot \id) \circ \Delta \displaybreak[0]\\
	&= \e_\star^{tL} \circ (\e_\star^{sL} \ot \id \ot \id) \circ (\id \ot \beta \ot \id) \circ (\Delta \ot \Delta) \circ (S \ot \id) \circ \Delta \displaybreak[0]\\
	&= (\e_\star^{sL} \ot \e_\star^{tL}) \circ \Lambda \circ (S \ot \id) \circ \Delta \\
	&= \e_\star^{(t+s)L} \circ (S \ot \id) \circ \Delta = F_{t+s}.
\end{align*}
The continuity of $F_t$ follows from the continuity of $\e_\star^{tL}$ and differentiating gives us the generator $\sigma = L \circ (S \ot \id) \circ \Delta$.

Analogously one concludes that the functionals $\e_\star^{tL} \circ (\id \ot S) \circ \Delta$ constitute a continuous convolution semigroup with generator $L\circ(\id \ot S)\circ\Delta$. But this equals $\sigma$ due to (\ref{eq:sigma}).
\end{proof}

\begin{Thm}
Every additive deformation on a braided Hopf algebra provides a family of deformed antipodes $(S_t)_{t \in \rR}$ with
\begin{align*}
	S_t = S \star \e_{\star}^{-t\sigma},
\end{align*}
where $\sigma=L \circ (\id \ot S)\circ \Delta$.
\end{Thm}
\begin{proof}
For $F_t=\e_{\star}^{tL}\circ(\id \ot S)\circ \Delta$, we get
\begin{align*}
	&\mu_t \circ \bigl(\id \ot (S \star F_{-t}) \bigr) \circ \Delta \\
	& \quad = (\e_\star^{tL} \ot \mu) \circ \Lambda \circ (\id \ot S \ot F_{-t}) \circ \Delta^{(3)} \displaybreak[0]\\
	& \quad = (\e_\star^{tL} \ot \mu) \circ (\id \ot \beta \ot \id) \circ (\Delta \ot \Delta) \circ (\id \ot S \ot F_{-t}) \circ \Delta^{(3)} \displaybreak[0]\\
	& \quad = (\e_\star^{tL} \ot \mu) \circ (\id \ot \beta \ot \id) \circ (\id \ot \id \ot \beta) \circ (\id \ot \id \ot S \ot S \ot F_{-t}) \circ \Delta^{(5)} \displaybreak[0]\\
	& \quad = (\e_\star^{tL} \ot \id) \circ (\id \ot \beta) \circ (\id \ot \eins \ot \id) \circ (\id \ot S \ot F_{-t}) \circ \Delta^{(3)} \displaybreak[0]\\
	& \quad = (\e_\star^{tL} \ot \eins) \circ (\id \ot S \ot F_{-t})\circ \Delta^{(3)} \\
	& \quad = \eins \circ (F_t \ot F_{-t}) \circ \Delta = \eins\delta.
\end{align*}
\end{proof}
\begin{Cor}
The deformed antipodes $S_t$ of a braided Hopf algebra with additive deformation $\mu_t$ and generator $L$ has the properties
\begin{itemize}
	\item[$(i)$] $S_t(\eins) = \eins$,
	\item[$(ii)$] $S_t \circ \mu_{-t} = \mu_t \circ (S_t \ot S_t) \circ \beta$,
	\item[$(iii)$] $\Delta \circ S_{t+r} = (S_t \ot S_r) \circ \beta \circ \Delta$,
	\item[$(iv)$] if $\calB$ is commutative or cocommutative, we get $S_t \circ S_{-t} = \id$ and
	\item[$(v)$] if we have a braided Hopf $*$-algebra, $S_{-t} \circ * \circ S_t \circ * = \id$ is fulfilled.
\end{itemize}
\end{Cor}
The proof of this corollary is quite similar to the proof in the trivially braided case, see \cite{Preprint}.


\section{Schoenberg Correspondence on Braided $*$-Bialgebras}\label{secSchoen}

In this section we prove the following theorem, which generalizes theorem 2.1.11 of \cite{Wirth} and theorem 2.1 of \cite{FSS}.
\begin{Thm} $($Schoenberg correspondence for additive deformations$)$\label{Sa:Schoenberg_def}
Let $\calB$ be a braided $*$-bialgebra with an additive deformation $(\mu_t)_{t\in \ds R}$ and let $\psi: \calB \to \ds C$ be a hermitian, $\beta$-invariant linear functional with $\psi(\eins)=0$. Then the following two statements are equivalent:
\begin{itemize}
\item[$(i)$] $\varphi_t := \e_\star^{t\psi}$ is a state on $\calB_t$ for all $b \in \calB, t \geq 0$, i.e.\ $\varphi_t (\eins) = 1$ and $\varphi_t \circ \mu_t \, (b^* \ot b) \geq 0$,
\item[$(ii)$] $(\psi \circ \mu + L) (b^* \ot b) \geq 0$ for all $b \in \ker \delta$.\footnote{We say that $\psi$ is $L$-conditionally positive in this case.}
\end{itemize}
\end{Thm}
\begin{proof}[Proof of $(i)\Rightarrow (ii)$]
The function $t \mapsto \varphi_t \circ \mu_t \, (c^* \ot c)$ is positive for $t \geq 0$. For $c \in \ker \delta$ this function vanishes at $0$, since
\begin{align*}
\varphi_0 \circ \mu_0 \, (c^*\ot c) = (\delta \ot \delta) (c^* \ot c) = \betrag{\delta(c)}^2 = 0.
\end{align*}
So the derivative $\left.\frac{\rd}{\rd t}\bigl(\varphi_t\circ\mu_t (b^*\ot b)\bigr)\right|_{t=0}=(\psi\circ\mu+L)(b^*\ot b)$ must be positive in this case.
\end{proof}

The aim of the remainder of this short section is to prove the converse implication.
For a vector space $V$ turn the set $\overline{V}:=\gklammer{\overline{v}:v\in V}$ into a vector space by defining $\overline{v}+\lambda\overline{w}:=\overline{v+\overline{\lambda}w}$.

Now let $(\calC, \Delta, \delta)$ be a $\beta$-braided $*$-coalgebra.\footnote{Here we say $\beta$-braided coalgebra and mean $(\calC, \Delta, \delta, \beta)$ is a braided coalgebra, since we have to distinguish different braidings on the same coalgebra.}
Then $*$ can be interpreted as a linear map from $\calC$ to $\overline{\calC}$ and from $\overline{\calC}$ to $\calC$. We define  $\overline{a \ot b} := \overline{a} \ot \overline{b}$ and set
\begin{align*}
\overline{\Delta} &:= (* \ot *) \circ \tau \circ \beta^{-1} \circ \Delta \circ *, && \text{i.e.\ } & \overline{\Delta} (\overline{c}) &= \overline{\Delta(c)},\\
\overline{\delta} &:= \delta \circ *, && \text{i.e.\ } & \overline{\delta} (\overline{c}) &= \overline{\delta(c)},\\
\overline{\beta} &:= (* \ot *) \circ \tau \circ \beta^{-1} \circ (* \ot *) \circ \tau, && \text{i.e.\ } & \overline{\beta} (\overline a \ot \overline b) &= \overline{\beta (a \ot b)}.
\end{align*}
Then $\bigl( \overline{\calC}, \overline{\Delta}, \delta \bigr)$ is a $\overline{\beta}$-braided $*$-coalgebra and $\bigl( \overline{\calC} \ot \calC, (\id \ot \tau \ot \id) \circ (\overline{\Delta} \ot \Delta), \overline{\delta} \ot \delta \bigr)$ is a usual $*$-coalgebra, i.e.\ $\tau_{2,2}$-braided $*$-coalgebra.

We call a linear mapping $\calC\ot \calC\to \ds C$ \emph{bilinear form} on $\calC$ and a linear mapping $\overline{\calC}\ot \calC\to\ds C$ \emph{sesquilinear form} on $\calC$. For a bilinear form $K$ we define the corresponding sesquilinear form $\widetilde{K}:=K\circ(*\ot\id)$. This is a bijection of bilinear forms and sesquilinear forms on $\calC$.

\begin{Lem}
Let $\star$ be the convolution of bilinear forms w.r.t.\ the comultiplication $\Lambda = (\id \ot \beta \ot \id) \circ (\Delta \ot \Delta)$ on $\calC \ot \calC$ and let $\circledast$ be the convolution of sesquilinear forms w.r.t.\ the comultiplication $(\id \ot \tau \ot \id) \circ (\overline{\Delta} \ot \Delta)$ on $\overline{\calC}\ot \calC$. For two bilinear forms $M$ and $K$ on the $\beta$-braided $*$-coalgebra $\calC$ the following is fulfilled. If $M$ is $\beta$-invariant, we have
\begin{align*}
\widetilde{M \star K} = \widetilde{M} \circledast \widetilde{K}.
\end{align*}
\end{Lem}

\begin{proof}
\begin{align*}
&\widetilde{M \star K} = (M \star K) \circ (* \ot \id) \\
& \quad = (M \ot K) \circ (\id \ot \beta \ot \id) \circ (\Delta \ot \Delta) \circ (* \ot \id) \displaybreak[0]\\
& \quad = (M \ot K) \circ (\id \ot \beta \ot \id) \circ (\beta \ot \id \ot \id) \circ (* \ot * \ot \id \ot \id) \circ (\tau \ot \id \ot \id) \circ (\overline{\Delta} \ot \Delta) \displaybreak[0]\\
& \quad = K \circ (\id \ot M \ot \id) \circ (* \ot * \ot \id \ot \id) \circ (\tau \ot \id \ot \id) \circ (\overline{\Delta} \ot \Delta) \displaybreak[0]\\
& \quad = (M \ot K) \circ (\id \ot \tau \ot \id) \circ (\tau \ot \id \ot \id) \circ (* \ot * \ot \id \ot \id) \circ (\tau \ot \id \ot \id) \circ (\overline{\Delta} \ot \Delta) \displaybreak[0]\\
& \quad = (M \ot K) \circ (* \ot \id \ot * \ot \id) \circ (\id \ot \tau \ot \id) \circ (\overline{\Delta} \ot \Delta) \\
& \quad = \widetilde{M} \circledast \widetilde{K}.
\end{align*}
\end{proof}
With this lemma we get for a $\beta$-invariant bilinear form $K$ on $\calC$
\begin{align*}
\e_{\star}^{tK}(c^*\ot c)= \e_{\circledast}^{t\widetilde{K}}(\overline{c}\ot c)
\end{align*}
so the following is now a direct consequence of the Schoenberg correspondence for sesquilinear forms on coalgebras due to \textsc{Schürmann} \cite{schurman3}.
\begin{Lem}
Let $K$ be a $\beta$-invariant, hermitian bilinear form on a the $\beta$ braided $*$-coalgebra $\calC$. Then the following two statements are equivalent:
\begin{itemize}
\item $\e_{\star}^{tK}(c^*\ot c)\geq0$ for all $c\in \calC,t\geq0$,
\item $K(c^*\ot c)\geq0$ for all $c\in\ker\delta$.
\end{itemize}
\end{Lem}

With this we are able to prove the Schoenberg correspondence.
\begin{proof}[Proof of Theorem \ref{Sa:Schoenberg_def}, $(ii)\Rightarrow (i)$]
Let $L$ be the generator of the additive deformation $(\mu_t)_{t\in\ds R}$ and define $K:=\psi\circ\mu+L$, which is a hermitian, conditionally positive bilinear form on the $\beta$-braided $*$-bialgebra $\calB$. With the previous lemma we conclude
\begin{align*}
0\leq \e_{\star}^{tK}(c^*\ot c)=\e_{\star}^{t\psi\circ\mu+tL}(c^*\ot c)=\e_{\star}^{t\psi\circ\mu} \star \e_{\star}^{tL}(c^*\ot c)&=\e_\star^{t\psi}\circ (\mu\star \e_{\star}^{tL})(c^*\ot c)\\
&=\varphi_t\circ\mu_t (c^*\ot c),
\end{align*}
since $(\psi\circ\mu)\star L=\psi\circ(L\star\mu)=\psi\circ(\mu\star L)=L\star (\psi\circ\mu)$ and $\mu$ is a coalgebra homomorphism. 

From $\psi(\eins)=0$ it follows directly that $\e_\star^{t\psi}(\eins)=\e^{t\psi(\eins)}=\e^0=1$, since $\Delta(\eins)=\eins\ot\eins$.
\end{proof}

\section{Examples}\label{sec:examples}

Let $\calC$ be a coalgebra. An element $c \in \calC$ is called primitive, if
\begin{align*}
\Delta(c)= c \ot \eins + \eins \ot c.
\end{align*}
It follows directly that $\delta(c)=0$ for every primitive element $c$.

\begin{Prop}\label{prop:primitive}
Let $\calB$ be a $\beta$-braided bialgebra with additive deformation $\mu_t=\mu\star e_\star^{tL}$ and $a,b\in \calB$. If $a$ and $b$ are primitive, we have
\begin{align*}
\mu_t (a \ot b)= ab + tL(a \ot b) \eins.
\end{align*}
\end{Prop}

\begin{proof}
First let us calculate the coproduct.
\begin{align*}
\Lambda (a\ot b)&=(\id\ot\beta\ot\id )\bigl(\Delta (a)\ot\Delta (b)\bigr)\\
&=(\id\ot\beta\ot\id ) \bigl( (a\ot\eins+\eins\ot a)\ot 
(b\ot\eins+\eins\ot b) \bigr)\\
&=(\id\ot\beta\ot\id )
(a\ot\eins\ot b\ot\eins
+a\ot\eins\ot \eins\ot b
+\eins\ot a\ot b\ot\eins
+\eins\ot a\ot \eins\ot b)\\
&=a\ot b\ot\eins\ot\eins
+a\ot\eins\ot \eins\ot b
+\eins\ot \beta (a\ot b)\ot\eins
+\eins\ot\eins\ot a\ot b
\end{align*}
Since $L(\eins\ot c)=L(c\ot\eins)=0$ for all $c\in \calB$ and $\delta(b)=\delta(a)=0$, we get
\begin{align*}
\e_\star^{tL}(a\ot b)
&=(\delta\ot\delta+tL+t^2/2 L\star L + \cdots)(a\ot b)=tL(a\ot b)\\
\e_\star^{tL}(a\ot \eins)&= \e_\star^{tL}(\eins\ot b)=0\\
\e_\star^{tL}(\eins\ot \eins)&=(\delta\ot\delta)(\eins \ot \eins)=1.
\end{align*}
It follows that
\begin{align*}
\mu_t (a\ot b)=(\mu \star \e_\star^{tL})\ (a\ot b)=ab+tL(a\ot b)\eins.
\end{align*}
\end{proof}

Consider the polynomial algebra $\widetilde{\calB}:=\cC \langle x,x^* \rangle$ in two non-commuting adjoint indeterminates. For a monomial $M$ we define the grade $g(M)$ as the degree of the monomial $M$. Then
\begin{align*}
	\beta(M\ot N):=(-1)^{g(M)g(N)} N \ot M
\end{align*}
for monomials $M, N$ defines a braiding on $\cC \langle x,x^* \rangle$, which is a symmetry, i.e.\ $\beta^2=\id\ot\id$. So $\beta$-invariance of a map is equivalent to $\beta^{-1}$-invariance. It is easily checked that the multiplication is $\beta$-invariant. This is turned into a  $\beta$-braided Hopf $*$-bialgebra by defining comultiplication, counit and antipode on the generators as
\begin{align*}
	\Delta \bigl( x^{(*)} \bigr) = x^{(*)} \ot \eins + \eins \ot x^{(*)},
 \quad \delta \bigl( x^{(*)} \bigr) = 0,
\quad S(x^{(*)})=-x^{(*)}
\end{align*}
and extending them as algebra homomorphisms, resp.\ anti-homomorphism in the case of $S$. The ideal $I$ generated by elements of the form $xx^*+x^*x$ is a coideal.
One has to show $\delta(I)=0$, which is obvious, and $\Delta(I)\subset I\ot \widetilde{\calB} + \widetilde{\calB}\ot I$. Therefore we calculate
\begin{align*}
\Delta(xx^*) = \Delta(x)\Delta(x^*)
&=xx^* \ot \eins + x \ot x^* + \beta (x \ot x^*) + \eins \ot xx^*\\
&=xx^* \ot \eins + x \ot x^* - x^* \ot x + \eins \ot xx^*
\end{align*}
and analogously $\Delta(x^*x)=x^*x \ot \eins + x^* \ot x - x \ot x^* + \eins \ot x^*x$. Combining these two equations, we get 
\begin{align*}
\Delta(xx^*+x^*x)=(xx^*+x^*x)\ot \eins+\eins\ot(xx^*+x^*x)\in I\ot \widetilde{\calB} + \widetilde{\calB} \ot I.
\end{align*}
Furthermore, we have $\beta(I\ot \widetilde{\calB} + \widetilde{\calB}\ot I)\subset I\ot \widetilde{\calB} + \widetilde{\calB}\ot I$, so $\calB:=\widetilde{\calB}/I$ is also a braided Hopf $*$-algebra. A hermitian 2-cocycle on $\calB$ is given by $L(x^* \ot x) = 1$ and $L(M\ot N)=0$ for all other monomials. We want to show that it is commuting and $\beta$-compatible.

We use the following general proposition.

\begin{Prop} Let $\calB$ be a braided bialgebra and $\beta$ be a symmetry, i.e.\ $\beta\circ\beta =\id\ot\id$. The equations
\begin{align*}
\beta\circ\Delta(a)=\Delta(a)\quad\text{and}\quad\beta\circ\Delta(b)=\Delta(b) \quad \text{imply} \quad \beta\circ\Delta(ab)=\Delta(ab)
\end{align*}
for all $a,b\in \calB$. In particular, $\calB$ is cocommutative, if $\calB$ is generated by primitive elements.
\end{Prop}

\begin{proof}
\begin{align*}
&\beta\circ\Delta\circ\mu (a\ot b)\\
& \quad =\beta\circ(\mu\ot\mu)\circ(\id\ot\beta\ot\id) \circ (\Delta\ot\Delta)(a\ot b)\\
& \quad =(\mu\ot\mu)\circ\underbrace{(\id\ot\beta\ot\id)\circ(\beta\ot\beta)\circ (\id\ot\beta\ot\id )}_{\beta_{2,2}}\circ (\id \ot \beta \ot \id ) \circ (\Delta \ot \Delta )(a\ot b) \displaybreak[0]\\
& \quad =(\mu\ot\mu)\circ(\id\ot\beta\ot\id)\circ(\beta\ot\beta)(\Delta(a) \ot \Delta(b)) \displaybreak[0]\\
& \quad =(\mu\ot\mu)\circ(\id\ot\beta\ot\id)\circ(\Delta \ot \Delta )(a\ot b)\\
& \quad=\Delta\circ\mu (a\ot b)
\end{align*}
The second statement is a direct consequence of the first, since for a primitive element $a\in \calB$
\begin{align*}
\beta\circ\Delta(a)=\beta (a\ot\eins+\eins\ot a)=\eins\ot a +a\ot \eins=\Delta(a)
\end{align*}
and finite products of generators span $\calB$.
\end{proof}
In our example $\beta$ is a symmetry and $\calB$ is generated by primitive elements, so $\calB$ is cocommutative. Then also $\calB \ot \calB$ is cocommutative, as
\begin{align*}
\beta_{2,2}\circ\Lambda = (\id \ot \beta \ot \id ) \circ (\beta \ot \beta) \circ (\id \ot \beta^2 \ot \id) \circ (\Delta\ot\Delta) = \Lambda.
\end{align*}
So $L\star\mu=\mu\star L$ is fulfilled. To see that $L$ is $\beta$-invariant, we only need to calculate
\begin{align*}
(L \ot \id) \circ (\id \ot \beta) \circ (\beta \ot \id) (M \ot x^* \ot x) = (-1)^{2g(M)}M = M = (\id \ot L) (M \ot x^* \ot x) 
\end{align*}
for all monomials $M$, as $L$ vanishes for other terms. $L$ is obviously hermitian.  We have now completed showing that $L$ is the generator of an additive $*$-deformation.

We calculate $\mu_t=(\mu\ot e_\star^{tL})\circ\Lambda$. First we know from proposition \ref{prop:primitive} that 
\begin{align}\label{eq:ex:commutator}
\mu_t (x^*\ot x)=\mu (x^*\ot x)+tL(x^*\ot x)\eins=-xx^*+t\eins=-\mu_t (x\ot x^*)+t\eins,
\end{align}
since $x$ and $x^*$ are primitive.
Next notice that $\mu_t(M\ot N)$ can only differ from $\mu(M\ot N)=MN$, when $M$ contains a factor $x^*$ and $N$ contains a factor $x$. With these two facts we know $\mu_t$, because of associativity. Let $M=x^{m_1}(x^*)^{m_2}$ and $N=x^{n_1}(x^*)^{n_2}$ with $m_1, m_2, n_1, n_2\in\mathds{N}$ be two monomials. Write $m=m_1+m_2$ and $n=n_1+n_2$. Then
\begin{align*}
\mu_t (M\ot N)&=
\mu_t\circ \left( \mu^{(m)}\ot\mu^{(n)} \right) \left( (x^{\ot m_1}\ot (x^*)^{\ot m_2})\ot (x^{\ot n_1}\ot (x^*)^{\ot n_2}) \right)\\
&=\mu_t\circ \left( \mu_t^{(m)}\ot\mu_t^{(n)} \right) \left( (x^{\ot m_1}\ot (x^*)^{\ot m_2})\ot (x^{\ot n_1}\ot (x^*)^{\ot n_2}) \right)\\
&=\mu_t^{(m+n)} \left( x^{\ot m_1}\ot (x^*)^{\ot m_2}\ot x^{\ot n_1}\ot (x^*)^{\ot n_2} \right)
\end{align*}
Now one can use (\ref{eq:ex:commutator}) to calculate this.
The $*$-algebra $\calB_t=(\calB,\mu_t)$ is isomorphic to the $*$-algebra $\calA_t$ generated by $a,a^{*}$ and $\eins$ with the relation $aa^{*}+a^{*}a=t\eins$. The map $a\mapsto x$, $a^*\mapsto x^*$ can be extended as an algebra homomorphism $\widetilde{\Phi}_t:\mathds{C}\sklammer{a,a^*}\rightarrow \calB_t$. Since the relation is respected, i.e.\
\begin{align*}
\widetilde{\Phi}_t (aa^*+a^* a)=\mu_t (x\ot x^*+x^*\ot x)=t=\widetilde{\Phi}_t(t\eins),
\end{align*}
we get an algebra homomorphism $\Phi_t:\calA_t\rightarrow \calB_t$. It is clear from our considerations on $\mu_t$ that this is an isomorphism as it maps the vector space basis $\gklammer{a^k(a^*)^l \mid k, l\in\mathds{N}}$ of $\calA_t$ to the vector space basis $\gklammer{x^k(x^*)^l \mid k,l\in\mathds{N}}$ of $\calB_t$. 

Since $0$ is a hermitian, $L$-conditionally positive linear functional vanishing at $\eins$, the exponential $e_\star^{t0}=\delta$ is a state on every $\calB_t$. Note that this is less trivial than it seems at a first glance because e.g.\ $\delta\bigl(\mu_t(x^*\ot x)\bigr)=\delta(-xx^*+t\eins)=t$.

For every $q\neq 0$ there is a unique braiding $\beta_q$ on the algebra $\cC \sklammer{x,x^*}$ of two non-commuting, adjoint indeterminates s.t.\ 
\begin{itemize}
\item  $\cC \sklammer{x,x^*}$ is a $\beta_q$-braided $*$-algebra
\item $\beta_q$ is defined on the generators in the following way:
  \begin{align*}
    \beta_q (x \ot x)   &= q~ x \ot x       & \beta_q(x \ot x^*)     &= q~ x^* \ot x       \\
    \beta_q (x^* \ot x) &= q^{-1}~ x \ot x^* & \beta_q ( x^* \ot x^*) &= q^{-1}~ x^* \ot x^* 
  \end{align*}
\end{itemize}
These equations determine $\beta_q$ on all pairs of monomials due to the compatibility of the unit and the multiplication. There exists a compatible Hopf $*$-algebra structure s.t.\ $\Delta(x^{(*)})=x^{(*)}\ot \eins+\eins\ot x^{(*)}$ and the ideal $I_q$ generated by elements of the form $xx^*-qx^*x$ is a coideal with $\beta_q(I\ot \widetilde{\calB} + \widetilde{\calB}\ot I)\subset I\ot \widetilde{\calB} + \widetilde{\calB}\ot I$. Dividing by this biideal yields a $\beta_q$-braided Hopf $*$-algebra $\calB_q$ with two $q$-commuting, primitive, adjoint generators. Note that for $q=-1$ the previous example is obtained. But for $q\neq\pm 1$ a multiplication $\mu_t$ on $\calB_q$ s.t.\ 
\begin{equation*}
  \mu_t(x \ot x^* - q x^* \ot x)=t\eins
\end{equation*}
cannot be $\beta_q$-compatible, as it would follow that
\begin{equation*}
  (\mu_t \ot \id) \circ 
\underbrace{(\id \ot \beta_q) \circ (\beta_q \ot \id)}_{(\beta_q)_{1,2}} 
\bigl(x \ot (x \ot x^* - q~ x^* \ot x)\bigr)=q^2~t \eins \ot x,
\end{equation*}
but
\begin{equation*}
  \beta_q \circ (\id \ot \mu_t)\bigl(x \ot(x \ot x^* - q~ x^* \ot x)\bigr)=t\eins \ot x.
\end{equation*}
So this can only work for the considered cases $q=\pm 1$.
Still our version of the Schoenberg correspondence applies to the braided Hopf $*$-algebras $(\widetilde{\calB},\beta_q)$ and $(\calB_q,\beta_q)$.
\newpage

\bibliographystyle{alpha}
\linespread{1.25}

\addcontentsline{toc}{section}{Bibliography}
\bibliography{literatur}

\end{document}